\documentclass[12pt, A4]{article}
\usepackage{amsmath,amssymb,amsfonts,mathrsfs,hyperref,microtype, amsthm}
\usepackage{geometry}
\usepackage[framed,numbered]{matlab-prettifier}
\usepackage[pdftex]{graphicx}
\usepackage{eufrak}
\usepackage{graphicx,color}
\usepackage[T1]{fontenc}
\usepackage{rotating}
\usepackage{booktabs}
\usepackage{mathtools}
\usepackage{float}
\usepackage{breqn}
\usepackage{graphicx}
\usepackage{caption}
\usepackage{multirow}
\usepackage{authblk}
\usepackage{comment}
\usepackage{siunitx}

\usepackage[T1]{fontenc}
\usepackage{lmodern}

\newtheorem{thm}{Theorem}[section]
\newtheorem{rem}{Remark}[section]
\newtheorem{mydef}{Definition}[section]
\newtheorem{lem}{Lemma}[section]
\newtheorem{prop}{Proposition}[section]
\newtheorem{cor}{Corollary}[section]

\newtheorem{ex}{Example}[section]

\numberwithin{equation}{section}

\def \E{{\rm I\kern-0.16em E}}
\def\P{{\rm I\kern-0.16em P}}
\def\F{{\rm I\kern-0.16em F}}
\def\B{{\rm I\kern-0.16em B}}
\def\C{{\rm I\kern-0.46em C}}
\def\G{{\rm I\kern-0.50em G}}
\def\D{{\rm I\kern-0.50em D}}

\newcommand{\R}{\mathbb{R}}
\newcommand{\N}{\mathbb{N}}

\newcommand{\KK}{\mathbb{K}}

\newcommand{\CC}{\mathbb{C}}

\newcommand{\Ci}{\mathrm{i}}

\newcommand{\PSO}{\text{PSO}}

\newcommand{\Ker}{\text{Ker}}

\newcommand{\BF}{\mathbf{\mathcal{F}}}
\newcommand{\Ann}{\text{Ann}}

\usepackage{color} 
\newcommand{\gr}[1]{{\color{blue} #1}}
\newcommand{\gdr}[1]{{\color{red} #1}}

\def\ind{\mathrel{\hbox{\rlap{%
\hbox to 7.5pt{\hrulefill}}\raise6.6pt\hbox{\eka\char'167}}}}
\begin{document}
\title{
\bf Polynomial Stein operators: a noncommutative algebra perspective
}

\author{Ehsan Azmoodeh 
	\thanks{Department of Mathematical Sciences,
		University of Liverpool, Liverpool L69 7ZL, United Kingdom.
		E-mail: \texttt{ehsan.azmoodeh@liverpool.ac.uk}}, \quad 
	Dario Gasbarra
	\thanks{Department of Mathematics and Statistics, University of Helsinki,  B314, FI-00014 Helsinki, Finland. E-mail: \texttt{dario.gasbarra@helsinki.fi}}, \quad 
	Robert E$.$ Gaunt \thanks{ Department of Mathematics, The University of Manchester, M13 9PL Manchester,  United Kingdom. E-mail:\texttt{ robert.gaunt@manchester.ac.uk}}
}

\pagestyle{plain}
\date{}
\maketitle 

\vspace{-10mm}

\abstract 


In this paper, we make a novel connection between Stein's method and noncommutative algebra by viewing polynomial Stein operators (Stein operators with polynomial coefficients) as elements of the first Weyl algebra. Through this connection we study the algebraic structure of classes of polynomial Stein operators. In the case of the standard Gaussian distribution, we provide a complete description of the corresponding class of polynomial Stein operators by (i) identifying it as a vector space over $\R$ with an explicit given basis and (ii) by showing that this class is a principal right ideal of the first Weyl algebra generated by the classical Gaussian Stein operator $\partial -x$, with $\partial$ denoting the usual differential operator. We also study the characterising property of polynomial Stein operators for the standard Gaussian distribution, and give examples of general classes of polynomial Stein operators that are characterising, as well as classes that are not characterising unless additional distributional assumptions are made. By appealing to a standard property of Weyl algebras, we shown that the non-characterising property possessed by a wide class of polynomial Stein operators for the standard Gaussian distribution is a consequence of a general result that is perhaps surprising from a probabilistic perspective: the intersection between the class of polynomial Stein operators for any two target distributions with holonomic densities or holonomic characteristic functions is non-trivial.


\section{Introduction} 

We begin with the following definitions that play a pivotal role in this paper.
 The first definition concerning polynomial Stein operators is taken from our recent paper \cite{a-g-g-algebraic-stein-operators}. All random variables throughout the paper are real-valued.
\begin{mydef}\label{def:PSO}
Let $X$ be a (continuous) target random variable.	Let $T \ge1$. We say that a linear differential operator $S:=S_X = \sum_{t=0}^{T} p_t \partial^t $ acting on a suitable class $\mathcal{F}$ of functions is a \textbf{polynomial Stein operator} for $X$ if :
	\begin{itemize}
	 \item[(a)] $Sf\in L^1(X)$, 
	 \item[(b)] $\E \left[ S f(X) \right] =0$ for all $f \in \mathcal{F}$, and 
	 \item[(c)] the coefficients $p_t$, $t=0,1,\ldots,T$, are polynomial. 
	 \end{itemize}
	 By $\mathrm{PSO}_{\mathcal{F}}(X)$ we denote the set of all polynomial Stein operators, acting on a class $\mathcal{F}$ of functions, for the target random variable $X$. 
	 Hereafter, for the ease of notation, we write $\mathrm{PSO}_{\mathcal{F}}(X) = \mathrm{PSO}(X)$.
\end{mydef}


\begin{mydef}[\textbf{Characterising polynomial Stein operators}]\label{def:Charactrization}
 We say that an operator $S \in \mathrm{PSO}(Y)$ is a \textbf{characterising} Stein operator for the target random variable $Y$ if, for any given random variable $X$ such that $Sf \in L^1(\P_X)$ for all $f \in \mathcal{F}$, the condition
 \begin{equation*}
  \E \left[ Sf (X)\right]=0, \quad \forall \, f \in \mathcal{F}  
  \end{equation*}   
  implies $X  =_d Y$ (equality in distribution).
\end{mydef}

\begin{rem}
	{\rm Observe that $0\in\mathrm{PSO}(X)$ is a trivial (and non-characterising) Stein operator for the target random variable $X$.}
\end{rem}

\begin{rem}[\textbf{Note on the class of functions $\mathcal{F}$ I}]\label{remi}
	{\rm  Consider the Stein operator $S = \sum_{t=0}^{T} p_t \partial^t $ for the random variable $X$ supported on $I\subseteq\mathbb{R}$. In the Stein's method literature, a commonly used choice of test functions is the class $\mathcal{F}_{S,Y}$ which contains all functions $f\in C^T(I)$ such that $\mathbb{E}|X^j f^{(t)}(X)|<\infty$, $j=0,\ldots,m$, $T=0,\ldots,T$, where $m=\mathrm{max}_{0\leq t\leq T}\,\mathrm{deg}(p_t)$. 
    Examples of other suitable classes of functions and clarifications on the classes of functions used in this paper are given in Remark \ref{rem_cau} at the end of the Introduction.}
\end{rem}

Polynomial Stein operators lie at the heart of Stein's method \cite{stein} and the Nourdin-Peccati Malliavin-Stein method \cite{np09,n-p-book}, powerful techniques for proving quantitative probabilistic limit theorems. Indeed, the starting point of Stein's method for Gaussian approximation \cite{stein,chen,n-p-book} is the following characterisation of the standard Gaussian distribution, known as \emph{Stein's Lemma}. Consider the divergence operator $G=\partial-x$, where $\partial=\frac{\mathrm{d}}{\mathrm{d}x}$ denotes the usual differential operator. Then, a real-valued random variable $X$ is equal in distribution to $N\sim N(0,1)$ if and only if $\E[Gf(X)]=0$ for all absolutely continuous $f:\R\rightarrow\R$ such that $\E|f'(N)|<\infty$. 

Motivated by Stein's Lemma, \cite{stein} introduced the so-called Stein equation $G f(x)=h(x)-\mathbb{E}[h(N)]$, where $h:\mathbb{R}\rightarrow\mathbb{R}$ is a test function. Evaluating both sides of the Stein equation at a random variable of interest $W$ and taking expectations now yields the transfer principle
\begin{equation}\label{transfer}
\mathbb{E}[h(W)]-\mathbb{E}[h(N)]=\mathbb{E}[Gf_h(W)],    
\end{equation}
where $f_h(x)=-e^{x^2/2}\int_{-\infty}^x\{h(t)-\mathbb{E}[h(N)]\}e^{-t^2/2}\,dt$ solves the Stein equation  $G f(x)=h(x)-\mathbb{E}[h(N)]$. Many important probability metrics, including the Kolmogorov, total variation and Wasserstein distances, can be expressed as integral probability metrics of the form $\sup_{h\in\mathcal{H}}|\mathbb{E}[h(W)]-\mathbb{E}[h(N)]|$ where the supremum is taken over some class of functions $\mathcal{H}$, e.g., taking $\mathcal{H}_{\mathrm{K}}=\{\mathbf{1}_{\cdot\leq z}\,|\,z\in\mathbb{R}\}$ induces the Kolmogorov distance. Thus the problem of bounding the distance between the distributions of $W$ and $N$ with respect to an integral probability metric reduces to bounding the quantity $\mathbb{E}[Gf_h(W)]$ for a given $h\in\mathcal{H}$ and then taking the supremum over all $h\in\mathcal{H}$. This procedure has proven to be remarkably effective in the context of Gaussian approximation, partly on account of the fact that the quantity $\mathbb{E}[Gf_h(W)]$ is expressed as an expectation of just the random variable $W$; the Gaussianity of the limit random variable $N$ is encoded in the operator $G$. More generally, this procedure can be applied for other limit distributions, in which a suitable Stein operator $S$ (that is preferably shown to be characterising) takes the place of Stein's classical operator $G$.

 Stein operators have classically been used as part of Stein's method to establish explicit bounds in distributional approximations, finding applications in fields as diverse as random graph theory \cite{bhj92}, 
 algebra \cite{fr22}
 and number theory \cite{h09}.
In recent years, there has also been a growing trend in which Stein operators have found applications beyond proving quantitative limit theorems, with examples including relaxing the Gaussian assumption in shrinkage \cite{fathi}, new tests for goodness-of-fit \cite{steinstat} and new methods for parameter estimation \cite{ebner}. 

Whilst researchers have been successfully applying Stein's method for over 50 years, mysteries about the method itself abound. 
One such mystery (Q1) is: why does the Stein operator $G=\partial-x$ work so well for normal approximation? This is true for most other distributions in that, in the current literature, only one or two Stein operators tend to be useful in the context of distributional approximation for a given target distribution, even though there are infinitely many Stein operators for a given distribution. Related to this question, the following question is natural: can a description be given for a typical element of the class of polynomial Stein operators for a given distribution (Q2)? Another mystery surrounds the characterising nature of Stein operators. It is well-known that the standard Gaussian Stein operator $G=\partial-x$ is characterising in the sense of Definition \ref{def:Charactrization}. However, does this remain true for all operators in the class $\PSO(N)$ (Q3)? More generally, one can ask about the characterising nature of Stein operators for general continuous probability distributions.
Within the Stein's method community, speculations include: `all polynomial Stein operators are characterising, but it can just be difficult to prove this for a given Stein operator' and `Stein operators are characterising if and only if the distribution is determined by its moments.' In Theorem \ref{thm:PSO_Intersection_NonTrivial}, we in fact see that the characterising property enjoyed by many Stein operators in the literature cannot be taken for granted (even for distributions that are determine by their moments): all classes of polynomial Stein operators for distributions with a holonomic density or holonomic characteristic function contain non-trivial elements that are not characterising.



In this paper, we provide, at least partial, answers to questions (Q1), (Q2) and (Q3). We do so by making a novel connection between Stein's method and noncommutative algebra. In particular, we view polynomial Stein operators as elements of the first Weyl algebra, the ring of linear differential operators with polynomial coefficients (in one variable). Weyl algebras enjoy a rich algebraic structure and their properties are well-understood; we refer the reader to \cite{Bjork-Book, Coutinho-Book, gw-Book, mr-Book87, Lam-Book, cf-Book} and references therein for introductions and comprehensive accounts of the theory. For example, it is well-known that Weyl algebras are simple integral domains, Noetherian (left/right) and Ore extensions. Two-sided ideals of Weyl algebras have plain structure; however, the structure of one-sided ideals is less obvious. Forasmuch as being Noetherian implies all one-sided ideals are finitely generated; however, more delicately, a famous theorem by Stafford \cite{Stafford-2generators} says that every one-sided (left/right) ideal of Weyl algebras can be generated by two elements. With this powerful theory at hand, we are able to address questions (Q1), (Q2) and (Q3) 
and consequently gain entirely new perspectives on Stein's method, some of which are perhaps surprising from a probabilistic viewpoint. 

For (Q1), we identify $\PSO(N)$ as a principle right ideal of the first Weyl algebra generated by Stein's classical operator $G=\partial-x$ (see Theorem \ref{prop:PSO=Ideal}). Various interpretations have been given as to why $G$ is a natural object in Gaussian analysis via Stein's method, such as the the generator approach to Stein's method of \cite{b88,barbour2,gotze} in which $(\partial-x)\partial$ is recognised as the generator of an Ornstein-Uhlenbeck process with standard Gaussian stationary distribution, and the canoncial Stein operator viewpoint of \cite{ley}. Here, we provide another such interpretation, which we consider to be fundamental and rather mathematically satisfying. Moreover, due to the aforementioned Stafford's Theorem of \cite{Stafford-2generators}, the fact that every one-sided (left/right) ideal of Weyl algebras can be generated by two elements implies that for a given target random variable $X$, the class $\PSO(X)$ can be generated by two Stein operators (see Proposition \ref{prop:PSO_General_Targets}). This result explains why, for a given continuous distribution, one would expect that only one or two polynomial Stein operators would prove to be most useful in the context of distributional approximation, despite the fact that there are infinitely many Stein operators for a given distribution.

We consider it to be a fundamental question as to whether the ideal $\PSO(X)$, for a given continuous random variable $X$, is principle. In Section \ref{sec3.3}, we consider this question for the random variables $H_p(N)$, $p\geq1$, where $H_p$ is the probabilist's Hermite polynomial of degree $p$. These random variables are of fundamental importance in Gaussian analysis and Malliavin calculus, and appear as target distributions in the asymptotic theory of U-statistics (see \cite[Section 4.4]{k-b-Ustatistics} and \cite[Chapter 3]{lee-Ustatistics}); however, the Malliavin-Stein method has yet to be adapted to these distributions for $p\geq3$. In Examples \ref{ex:H_2(X)IsPrincipal}, \ref{ex:H_3(X)IsNotPrincipal} and \ref{ex:H_4(X)IsNotPrincipal} , we prove that the ideal  $\mathrm{PSO}(H_2(N))$ is principle, whilst the ideals $\mathrm{PSO}(H_3(N))$ $\mathrm{PSO}(H_4(N))$ are not principle; we conjecture that $\mathrm{PSO}(H_p(N))$ is principle if and only if $p=1,2$ (see Remark \ref{remprinciple}). 

For (Q2), we provide a complete algebraic description of the class $\PSO(N)$ (Theorem \ref{prop:PSO=Ideal}). This theorem includes the aforementioned result that $\PSO(N)$ is a principle right ideal of the first Weyl algebra generated by the operator $G=\partial-x$.  We also show that $\PSO(N)$ is a vector space over $\R$ with the operators $S(k,t)=H_{k}\partial^t-H_{k+t}\in\PSO(N)$, $k\geq0$, $t\geq1$, forming a basis. Note that $S(0,1)=G$. 
Whilst methods exist for obtaining families of Stein operators for a given target distribution (see the flexible approach to Stein's density method of \cite{ley,mrs21} and our algorithmic approach \cite{a-g-g-algebraic-stein-operators}), this is the first time in the literature that a complete description has been given for the class of polynomial Stein operators for a given target distribution. 
Interestingly, the fact that $S(k,t)\in\PSO(N)$, which yields a useful integration by parts formula, was exploited by \cite{b86} in order to derive asymptotic expansions for smooth functions of sums of independent random variables. We also remark that the Stein operators $S_m = S(m-1,1)= H_{m-1}\partial - H_m $ and $ L_m=S(0,m)= \partial^m - H_m $, $m\geq1$, were also introduced as natural generalisations of the operator $G$ in the delightful paper \cite{gr05}. 

In Section \ref{sec3.1}, we complement the algebraic descriptions of $\PSO(H_p(N))$ given in Sections \ref{sec3.2} and \ref{sec3.3} by making a connection between Stein's method and holonomic function theory, which allows us to provide simple conditions under which $\PSO(X)$ is non-trivial for a given target random variable $X$. As an example, in Corollary \ref{cor3.1}, we show that $\PSO(h(N))\not=\{0\}$ for any $h\in\mathbb{R}[x]$ (the ring of polynomials with real-valued coefficients).

The characterising problem (Q3) is the subject of Sections \ref{sec4} and \ref{secnew}. In Proposition \ref{cor:Hermite-Coeff-Is-Characterizing}, we show that the polynomial Stein operators $S_m=H_{m-1}\partial-H_m$, $m\geq1$, are characterising in the sense of Definition \ref{def:Charactrization}. This is a consequence of a more general result that all polynomial Stein operators for the standard Gaussian distribution that are first order linear differential operators are characterising (see again Proposition \ref{cor:Hermite-Coeff-Is-Characterizing}). We remark that the Stein operators $(\partial-x)\partial^n$, $n\geq0$, that is the Stein operators for the Gaussian distribution with coefficients with polynomial degree one, are seen to be characterising as a direct consequence of Proposition 2.1 of our recent paper \cite{a-g-g-Stein-charactrization}. We also show that the Stein operators $S(k,t)=H_{k}\partial^t-H_{k+t}\in\PSO(N)$, $k,t\geq1$, are characterising amongst the class of symmetric and infinitely divisible distributions (Corollary \ref{barops}). We conjecture that the additional assumptions of symmetric and infinitely divisible distributions are artefacts of our proof, and that these assumptions can be disposed of. However, it turns out that there are non-trivial elements of $\mathrm{PSO}(N)$ that are not characterising unless additional distributional assumptions are made; in part (d) of Example \ref{ex:Moment_Assumption_Crucial}, we give a general recipe for constructing such non-characterising polynomial Stein operators for the Gaussian distribution. This non-characterising property of elements of the class $\mathrm{PSO}(N)$ is in fact a special case of the result that if the characteristic functions of $X_1$ and $X_2$ are holonomic (satisfy a linear homogeneous differential equation with polynomial coefficients) or if the density functions of $X_1$ and $X_2$ are holonomic, then $\mathrm{PSO}(X_1)\cap\mathrm{PSO}(X_2)\not=\{0\}$ (see Theorem \ref{thm:PSO_Intersection_NonTrivial}). From a probabilistic point of view this result may seem surprising, but it is in fact a consequence of basic   properties of Weyl algebras.  This demonstrates the power of the connection between Stein's method and noncommutative algebra, and in this paper we have likely only scratched the surface of what can be achieved through this connection. We hope this paper inspires both probabilists and algebraists to further explore this seemingly fruitful new window on Stein's method.

\begin{rem}[\textbf{Note on the class of functions $\mathcal{F}$ II}]\label{rem_cau}
	{\rm We require that the class $\mathcal{F}$ is \textit{invariant} under any polynomial differential operator, that is to say $S \mathcal{F} \subseteq \mathcal{F}$ for any differential operator $S$ with polynomial coefficients. For probabilistic purposes, the following natural choices for function space $\mathcal{F}$ are available: (a) $\mathcal{F}=\R[x]$ the ring of all polynomials with real coefficients; (b) $\mathcal{F} = C^\infty_c (\R)$ the space of all infinitely differentiable functions with compact support;  (c)     the class of Schwartz functions $\mathcal{F} = \mathcal{S} (\R): =  \{  f \in C^\infty (\R) \, : \, 
			\sup_{x\in \R} \vert x^\alpha  \partial^\beta f (x) \vert < \infty, \, \forall \, \alpha,\beta \in \N_0         \}$, which includes the class $C^\infty_c(\R)$; and (d) $\mathcal{F}=P(\R):=\{  f \in C^\infty (\R) \, : \, 
			\sup_{x\in \R} \vert (1+x^\alpha)^{-1}  \partial^\beta f (x) \vert < \infty \,\text{for all $\beta \in \N_0$ and some $\alpha\geq1$} 
            \}$ the space of infinitely differentiable functions with polynomial growth. Note that these families of functions are invariant under any polynomial differential operator. 
			The Schwartz class $\mathcal{F} = \mathcal{S} (\R)$ is applicable to all target distributions. When target distributions have all moments (such as distributions in Wiener chaoses or probability distributions with bounded support), the classes $\mathcal{F}=\R[x]$ and  $\mathcal{F}=P(\R)$ are also available.
The class	$\mathcal{F}=\R[x]$ is of particular interest for target distributions that are moment-determined. It is classical that the standard Gaussian distribution is moment-determined; see e.g. \cite[Lemma 3.1.]{n-p-book}. The class $\mathcal{F}=P(\R)$ is useful in practical implementations of Stein's method, containing the real and and imaginary parts of the function $f(x)=e^{\mathrm{i}tx}$, $t\in\mathbb{R}$, and ensuring that, for $f$ in this class, $\E|X^\alpha\partial^\beta f(X)|<\infty$ for all $\alpha,\beta\in\N_0$, if $X$ has all moments. To keep the statements of our results concise, when considering an arbitrary target random variable, the shorthand notation $\PSO(X)$ can be safely considered to represent $\PSO_{\mathcal{F}}(X)$, where $\mathcal{F} = \mathcal{S}(\R)$.
 However, if all moments of the target random variable(s) exist (as is the case for the Gaussian distribution), then the results also apply for $\mathcal{F}=\R[x]$,  $\mathcal{F}=P(\R)$, and the class $\mathcal{F}_{S,Y}$ defined in Remark \ref{remi}.
}
\end{rem}


\section{Auxiliary background} \label{se:Background}

\subsection{Elements of Weyl algebra $A_n(\KK)$}\label{sec:Weyl_Algebra}

In this section, we collect some relevant basic definitions and results from noncommutative algebra; for a comprehensive treatment of the subject, the reader can consult the classical textbooks \cite{mr-Book87, cf-Book, gw-Book, Lam-Book}. We begin by recalling some definitions from ring theory, which will be used throughout this paper. We then introduce some basic theory for Weyl algebras and holonomic functions that will be required in this paper. Throughout this section, $\KK$ denotes a commutative field of characteristic zero. For our purposes, either $\KK=\R$, or $\KK=\mathbb{C}$ (the field of complex numbers). 


\begin{mydef}[\textbf{Ring}]\label{def:Ring}
A ring is a triplet $(R, +, \cdot)$ in which $R$ is a non-empty set equipped with two binary operations $+$ and $\cdot$ called addition and multiplication, respectively, such that:
\begin{itemize}
\item[(i)]  $(R,+)$ is an abelian group.
\item[(ii)] $(R,\cdot)$ is a monoid.
\item[(iii)] Addition and multiplication are related to each other via distributive laws:  $(a+b)\cdot c = a \cdot  b + b\cdot c$ and $ a \cdot (b+c) = a \cdot b + a \cdot c$.
\end{itemize}
The neutral element of $R$ under the addition operator is called zero and denoted by $0_R=0$, whereas the neutral element under the multiplication operator is called one or the unit-element and is denoted by $1_R=1$. The inverse of an element $a \in R$ under the addition operator is denoted by $-a$, i.e., $a+ (-a)=0$. The ring $R$ is called commutative if $a\cdot b = b\cdot a$ for all $a,b \in R$, and otherwise we say $R$ is noncommutative.  We often omit the multiplication sign and write $ab$ instead of $a \cdot b$.
\end{mydef}

\begin{mydef}[\textbf{Ideal}]\label{def:Ideal}
A left (right) ideal $I$ of a ring $R$ is a subgroup of $(R,+)$ such that $RI \subseteq I$ $(IR \subseteq I)$. We also write $I \lhd_l R $ and $I \lhd_r R$ to denote a left/right ideal. A two-sided ideal $I$ -- often referred to simply as an ideal -- is a left and right ideal simultaneously, that is  $RI \subseteq I $ and $IR \subseteq I$. We denote it by $I \lhd R$.	
\end{mydef}

We now recall a number of further fundamental concepts from ring theory:

\begin{mydef}\label{def:Basic_Ring_Theory}
	Let $R$ be a ring (not necessarily commutative).
	\begin{itemize}
		\item[(a)] The ring $R$ is an integral domain  	if the product of two non-zero elements is non-zero, i.e., $ab\neq0$ for every $0 \neq a ,b \in R$. 
		\item[(b)] The ring $R$ is simple if $0$ and $R$ are the only two-sided ideals of $R$. 
		\item[(c)] The ring $R$ is a  principal left (right) ideal ring if every left (right) ideal is principal (or generated by only one element).
		\item[(d)] A left (right) ideal $I \lhd_l R$ ($I \lhd_r R$) is called maximal if there is no left (right) proper ideal $J$ containing $I$.
		\item[(e)] The ring $R$ is prime if for every $0 \neq a,b \in R$ then $a R b \neq 0$. A left (right) ideal $I \lhd_l R$ ($I \lhd_r R$) is called a prime ideal if the quotient $R/I$ is a prime ring.  
		\item[(f)] The ring $R$ is called left (right) Noetherian if it satisfies the left (right) ascending chain condition, 
		namely that every chain of left (right) ideals $ I_1 \subseteq I_2\subseteq I_3\subseteq \cdots  $ must terminate after a finite step, i.e., there exists $n\in\N$ such that $I_n=I_{n+1}=\cdots$. 
	\end{itemize}	
\end{mydef}

A further fundamental concept from ring theory that will be needed in this paper is that of modules.

\begin{mydef}[\textbf{Module}]\label{def:Module}
Let $R$ be a ring (not necessarily commutative). An abelian group $(M,+)$ is called a left $R$-module if there is a mapping $\bullet :  R \times M \to M$ such that for all $a,b \in R$ and for all $m,m' \in M$:
\begin{itemize}
	\item[(i)]  $a \bullet (m+m') = a\bullet m + a \bullet m'$.
	\item[(ii)] $ (a+b)\bullet m = a \bullet m + b \bullet m$.
	\item[(iii)] $(ab) \bullet m = a \bullet (b\bullet m)$.
	\item[(iv)]  $1\bullet m = m$, where $1 = 1_R$.
 	\end{itemize} 
In brief, we write ${}_R M$ to say $M$ is a left $R$-module. A right $R$-module  is defined similarly, and denoted by $M_R$. A left $R$-module ${}_R M$ is said to be finitely generated if there is a finite set $X=\{ x_i \, : \, i \in I\} \subseteq M$, say of the cardinality $ \vert I \vert =n$, so that for every element $m \in M$ there are $a_1,\ldots,a_n \in R$ such that $m = a_1\bullet x_1 +\ldots+a_n \bullet x_n$. In this case we write ${}_R M = {}_R \langle X \rangle = {}_R \langle x_1,\ldots,x_n \rangle$. If $n=1$, that is when ${}_R M$ is generated by only one element, it is called a cyclic module.
\end{mydef}

\begin{ex}
(a) Let $R$ be a ring. Every left ideal $I \lhd_l R $ with the ring multiplication is a left $R$-module. Similarly, every right ideal 	$I \lhd_r R $ with the ring multiplication is a right $R$-module. (b) Let $a_1,\ldots,a_n \in R$. Then $ {}_R \langle a_1,\ldots,a_n \rangle = \sum_{i=1}^{n} Ra_i = \{  \sum_{i=1}^{n} r_i a_i \, : \, r_i \in R, \, 1 \le i \le n \}$ is a finitely generated left $R$-module. 
\end{ex}

Before moving on to some basic theory for Weyl algebras and holonomic functions we introduce the notion of an annihilator of a subset of a module over a ring, which will be used throughout this paper and also in our definition of holonomic functions later in this section.

\begin{mydef}[\textbf{Annihilator}]\label{def:Annihilator}
Let ${}_R M$ be a left $R$-module. Let $X \subseteq M$ be a subset of $M$. The annihilator of $X$, denoted by $\mathrm{Ann}_{R}(X)$, consists of those elements $r \in R$ such that for all elements $x \in X$, $r\bullet x =0$. In set notation,
\begin{equation*}
\mathrm{Ann}_R (X)  = \{    r \in R  \, : \,  r \bullet x =0 \,\, \forall \, x \in X \}.
\end{equation*}
\end{mydef}

We are now ready to define the Weyl algebra.

\begin{mydef}[\textbf{Weyl algebras}]\label{def:Weyl_Algebra}
	Let $n \in \N$ be a natural number. The $n$th Weyl algebra $A_n (\KK)$ is an associative $\KK$-algebra (a vector space over the field $\KK$ equipped with a bilinear operator which is associative, that is to say $r(a\cdot b) = (ra)\cdot b = a \cdot (rb) $ for every $r \in \KK$ and every $a,b \in A_n (\KK)$)  with generators $x_n,\ldots ,x_n$, $y_1,\ldots ,y_n$ subject to the relations 
	\begin{equation*}
	[x_i,x_j]=[y_i,y_j]=0, \quad [x_i,y_j]=\delta_{ij}, \quad \forall \, 1 \le i,j\le n,
	\end{equation*}
	where $[a,b]=ab-ba$ is the commutator of the elements $a$ and $b$, and $\delta_{ij}$ denotes  the Kronecker delta.  Let $\KK[x_1,\ldots ,x_n]$ be the ring of polynomials in $n$ variables $x_1,\ldots ,x_n$ with coefficients in $\KK$. It is classical that the $n$th Weyl algebra $A_n (\KK)$ is isomorphic (as a $\KK$-algebra) to the ring of differential operators with polynomial coefficients in $n$ variables by identifying each $y_i$ with $\partial_i = \partial_{x_i}$. Clearly,  $A_n(\KK)$ is a noncommutative ring.
\end{mydef}

The following theorem (see \cite[Theorem 1.15]{mr-Book87}) collects some algebraic properties of Weyl algebras. The right Ore property described in part (b) of Theorem \ref{thm:Weyl_Algebra_Property} will allow us to provide an efficient proof of Theorem \ref{thm:PSO_Intersection_NonTrivial} that concerns the non-characterising nature of classes of polynomial Stein operators.

\begin{thm}\label{thm:Weyl_Algebra_Property}
	\begin{itemize}
		\item[(a)] The Weyl algebra $A_n (\KK)$ is a simple Noetherian (left/right) integral domain, and hence is a prime ring.
		\item[(b)] The Weyl algebra $A_n (\KK)$ satisfies the left/right Ore property, namely that $A_n (\KK) L \cap  A_n (\KK) S \neq \{0\}$ for every two non-zero elements $L, S \in A_n (\KK)$.
	\end{itemize}
\end{thm}

In order to formulate some further fundamental results on Weyl algebras we will require the following definitions.

\begin{mydef}\label{def:Holonomic-Module}
\begin{itemize}
	\item[(a)]  {\rm{\textbf{(Bernstein filtration)}}}.  Let $\nu \in \N_0$. Let $\mathscr{F}_\nu$ be a $\KK$-vector space which is generated by elements $\{ x^\alpha \partial^\beta \, :\,  \vert \alpha \vert + \vert \beta \vert \le \nu  \}$, where for a multi-index $\alpha=(\alpha_1,\ldots,\alpha_n)$ we let $\vert \alpha \vert = \sum_{i=1}^{n} \alpha_i$. Clearly, each $\mathscr{F}_\nu$ is a finite dimensional $\KK$-vector space, $\mathscr{F}_0 \subset \mathscr{F}_1 \subset \cdots$, $\cup_{\nu=0}^{\infty} \mathscr{F}_\nu = A_n (\KK)$ and $\mathscr{F}_\nu \mathscr{F}_k \subset \mathscr{F}_{\nu+k}$. The family $\mathscr{F} = \{   \mathscr{F}_\nu  \, : \, \nu \ge 0  \}$ is called the Bernstein filtration for the Weyl algebra $A_n(\KK)$.  The direct sum $\mathrm{gr} \left( A_n(\KK) \right) : =  \mathscr{F}_0 \oplus \frac{\mathscr{F}_1}{\mathscr{F}_0} \oplus  \frac{\mathscr{F}_2}{\mathscr{F}_1} \oplus \cdots $ is called the associated graded ring.
	\item[(b)]  {\rm{\textbf{(Good filtration)}}}. Let $M$ be a left $A_n(\KK)$-module. A filtration $\Gamma = (\Gamma_\nu \, : \, \nu \ge 0)$ is an increasing sequence of finite dimensional vector subspaces on $\KK$ such that $M = \cup_{\nu=0}^{\infty} \Gamma_\nu$ and $ \mathscr{F}_k \Gamma_\nu \subset \Gamma_{\nu+k}$ for every pair $k$ and $\nu$. For a given filtration $\Gamma$, the associated graded $\mathrm{gr} \left( A_n(\KK)\right)$-module is given as $\mathrm{gr}_{\Gamma} M : = \Gamma_0 \oplus \frac{   \Gamma_1}{\Gamma_0} \oplus \frac{\Gamma_2}{\Gamma_1} \oplus \cdots$. A filtration $\Gamma$ is called a good filtration if the associated graded $\mathrm{gr}_{\Gamma} M$ is a finitely generated graded $\mathrm{gr} \left( A_n(\KK)\right)$-module.
	\item[(c)]  {\rm{\textbf{(Bernstein dimension)}}}. Let $\Gamma$ be a good filtration on a finitely generated left $A_n(\KK)$-module $M$. Then, it is well-known that there is $d \ge 0$ and rational numbers $a_0,\ldots,a_d$ so that $\mathrm{dim}_K \Gamma_\nu = a_d \nu^d + \cdots+a_0$ for $\nu \gg 0$. 
	The integer $d=: d_\Gamma(M) = d(M)$ is called the Bernstein dimension, and it is well-known that $d$ is independent of the choice of good filtration.
	\item[(d)] {\rm{\textbf{(Holonomic module)}}}. A finitely generated left $A_n(\KK)$-module $M$ is called a holonomic module if $d(M)=n$. It is well-known that for every non-zero finitely generated left $A_n(\KK)$-module $M$ it holds $d(M)\ge n$. This is known as Bernstein's inequality.
	\end{itemize}
\end{mydef}

In studying the algebraic structure of $\mathrm{PSO}(N)$ in Section \ref{sec3.2} we will employ the following fundamental results.

\begin{thm}[\textbf{Stafford's theorems}]\label{thm:Stafford}
	\begin{itemize}
		\item[(a)] Every left (right) ideal in the $n$th Weyl algebra $A_n(\KK)$ may be generated by two elements -- see e.g. 
		\cite[Theorem 3.1]{Stafford-2generators}. 
		\item[(b)] Every holonomic $A_n$-module is cyclic, i.e. generated by one element -- see e.g. \cite[Corollary 2.6]{Coutinho-Book}.
	\end{itemize}
\end{thm}



The following example sets the scene for our brief review of relevant results from Holonomic function theory.

\begin{ex}
Let $M=C^\infty(\KK)$ be the vector space of infinitely many differentiable functions over the field $\KK$. Define the mapping $\bullet :   A_1(\KK) \times M  \to M$, where the action of $L=\sum_{t=0}^{T} p_t \partial^t$ on $f$ is defined by $L \bullet f = \sum_{t=0}^{T} p_t \partial^t f$. Recall that the coefficients $p_t$, $t=0,1,\ldots,T$, are polynomials in one variable. It is straightforward to see that ${}_{A_1(\KK)} M$ is a left $A_1(\KK)$-module.
\end{ex}

\subsection{Holonomic functions}


We begin by stating a definition of holonomic functions, making use of the terminology built up in this section. For an introduction to holonomic function theory the reader is referred to the excellent textbooks \cite{Stanley-Book, Flajolet-Book}.  

\begin{mydef}[\textbf{Holonomic function}]\label{def:Holonomic-Function}
	A function $f: {\rm{Dom}} (f) \subseteq \KK  \to \KK$ is called  holonomic (or $\partial$-finite, or $D$-finite) if it satisfies one of the following equivalent conditions:
	\begin{itemize}
		\item[(a)]   $\mathrm{Ann}_{A_1(\KK)} (f) : =  \{   L \in A_1(\KK)  \,  : \,  L \bullet f =0   \} \neq \{0\}$.
		\item[(b)]  The vector space spanned by $(\partial^t f :  t =0, 1, \ldots )$ is finite dimensional over the field of rational functions $\KK(x) = \big\{  \frac{p(x)}{q(x)} \, : \,  p, q \in \KK[x]   \big\}$ with the usual addition and multiplication operations on polynomials.
	\end{itemize}
\end{mydef}

 The class of holonomic functions enjoys rich closure properties. The following  result is borrowed from \cite[Theorem B2]{Flajolet-Book} and \cite[Theorem 1]{Non-Holonomic-Character}. These closure properties can allow for efficient proofs that given functions are holonomic when a direct application of the definition is not tractable, a point which is illustrated in our proof of Proposition \ref{holprop} that relies on the closure properties of holonomic functions. 

\begin{thm}\label{thm:Holonomic-Closure-Properties}
\noindent{(a)} The class of univariate holonomic functions is closed under the following mathematical operations: (i) sum, (ii) product, (iii) differentiation, (iv) indefinite integration, and (v) algebraic substitution.

\vspace{1mm}

\noindent{(b)} A holonomic function has only finitely many singularities.
\vspace{1mm}

\noindent{(c)} If the function $f(x_1,\ldots,x_n)$ is holonomic, then, for any $a$ and $b$, the function given by $ \int_{a}^{b} f(x_1,\ldots,x_{n-1},x_n) \,dx_n$ is also holonomic, if it is defined (\cite[Proposition 3.5]{Zeilberger90}).

\end{thm}

From Theorem \ref{thm:Holonomic-Closure-Properties}, we know that a holonomic function possesses only finitely many singularities. The asymptotic behaviour of a holonomic functions in a neighbourhood of one of its singularities is described by the following theorem. 

\begin{thm}[see \cite{Wasow-Book} Theorem 19.1, \cite{Non-Holonomic-Character} Theorem 2]\label{thmsing}
Let $f: {\rm{Dom}} (f)\subseteq\mathbb{R}\to \mathbb{R}$ be a holonomic function satisfying the ordinary differential equation (ODE)
\begin{equation}\label{eq:Holonomic-Property}
\sum_{t=0}^{T} p_t \partial^tf(x)=0.    
\end{equation}
Let $x_0$ be a singular point of $f$. Then there exists a sufficiently small neighborhood of $x_0$ such that the following asymptotic expansion holds:
\begin{equation*}
f(x)    \sim   \exp \left(  P \left(  x-x_0\right)^{-1/r} \right) (x-x_0)^\alpha \, \sum_{j=0}^{\infty}  Q_j (\log (x-x_0)) (x-x_0)^{j\beta},    \quad  x \to  x_0,
\end{equation*}
where $P$ is a polynomial, $r \in \N_0$, $\alpha \in \mathbb{C}$, $\beta \in \mathbb{Q}_{>0}$, and the $Q_j$ are a family of polynomials of uniformly bounded degree. The choice of the quantities $P,Q_j,r,\alpha, \beta$ depends upon the particular solution of the ODE \eqref{eq:Holonomic-Property}.
	\end{thm}

Let $L=\sum_{t=0}^{T} p_t \partial ^t$ be a linear differential operator with polynomial coefficients. The \textit{algebraic} (or \textit{formal}) adjoint of $L$ is defined as
\begin{equation}\label{eq:Algebraic-Adjoint}
L^* f  :=  \sum_{t=0}^{T} (-1)^t \partial^t(p_t \, f).
\end{equation}
It is classical \cite[page 211]{Ince-Book} that the operator $L$ and its algebraic adjoint $L^*$ are linked together by the so-called \textit{Lagrange identity}:  for sufficiently regular $f,g$,
\begin{equation}\label{eq:Lagrange-Identity}
g L f  - f L^* g = \frac{\partial}{\partial x} P_L (f,g),
\end{equation}
where $P_L (f,g)$ is the \textit{bilinear concomitant}, a homogeneous bilinear form which can be  written explicitly as 
\begin{equation}\label{eq:Bilinear-Concomitant}
\begin{aligned}
P_L (f,g) &= f \Big\{    p_1 g - \partial(p_2 g)+\cdots+(-1)^{T-1} \partial^{T-1} (p_T g)       \Big \} \\
& + f'  \Big\{    p_2 g - \partial(p_3 g)+\cdots+(-1)^{T-2} \partial^{T-2} (p_T g)       \Big \} \\
& \quad \vdots\\
&  + f^{(T-1)} p_T g.
\end{aligned}
\end{equation} 
The \textit{Green's} formula is obtained by taking integration over the interval $(a,b)\subseteq\mathbb{R}$ (where $-\infty\leq a<b\leq\infty$) of both sides of \eqref{eq:Lagrange-Identity}:
\begin{equation}\label{eq:Green-Formula}
\int_{a}^{b} g(x) L f(x) \,  dx   -  \int_{a}^{b}  f(x) L^* g(x) \, dx   =  P_L (f(x),g(x))\,\Big\vert_a^b
\end{equation}
whenever the integrals exist.   

\subsection{A gentle introduction to Malliavin operators}\label{sec:Malliavin}
For the scope of  our paper, it is enough  to define Malliavin operators in dimension $d=1$ algebraically,  without discussing their functional analytic extensions.
For a  state-of-the-art  exposition of Malliavin calculus  in full generality see \cite{n-p-book,n-n-book}.

\begin{mydef} 
In the univariate case, the Malliavin derivative  $D$, the divergence $\delta$ and its pseudo-inverse  
	$\delta^{-1}$ are defined as linear mappings acting  on the polynomial ring $\R[x]$, with
	\begin{align*} &D x^n=  \partial x^n = n x^{n-1}, \quad 
		\delta x^n =(x-\partial) x^n= x^{n+1}- n x^{n-1}, \quad \text{and}& \\ &
		\delta^{-1} 1 = 0 , \quad\delta^{-1} x= 1,\quad 
		\delta^{-1} x^n= x^{n-1} + (n-1) \delta^{-1} x^{n-2} = 
		\sum_{k=0}^{\lfloor \frac{n-1}2 \rfloor}
		\frac{ ( n-1)!\!!}{ (n-1-2k)!\!! } x^{n-1-2k}, &  
	\end{align*}
	where $n!\!!$ denotes the double factorial.  Also, it holds that $	\delta^{-1} \delta x^n=x^n$, and  $\delta \delta^{-1} x^n= x^n-\E \left[ X^n \right]$, with $X \sim N(0,1)$,
		where the standard Gaussian distribution has moment sequence $\E \left[ X^{n} \right]= (n-1)!\!!$ when $n\in 2\N$ and $\E \left[ X^{n} \right]=0$ otherwise. 
		
		Let $p\ge 2$. We define the $p$th Malliavin derivative and divergence operators as follows:  $D^p f = \partial^p f = f^{(p)}$, and $\delta^p f = \delta \left( \delta^{p-1} f\right) = \delta^{p-1} (\delta f)$.
\end{mydef}

\begin{prop}[\textbf{Gaussian integration by parts}] \label{prop:GIBP-d=1} 
Let $X \sim N(0,1)$. For $f,g\in \R[x]$,
	\begin{align}\label{lem:GIBP-d=1}
	\E\bigl[  f(X) D g(X) \bigr]= 	\E\bigl[  f(X) \partial  g(X) \bigr]= \E\bigl[  g(X) \delta f( X)  \bigr]. 
	\end{align}
\end{prop}

Hermite polynomials play a fundamental role in Malliavin calculus. For our purposes, it will suffice to provide the following definition of the Hermite polynomials that employs the divergence operator, along with the forthcoming fact that any every polynomial can be written as a finite sum of Hermite polynomials. Many further properties can be found in the books \cite{n-p-book,n-n-book}.

\begin{mydef}[\textbf{Hermite polynomials}]\label{def:Hermite_Polynomials} 
We define the univariate Hermite polynomials as $H_0(x)=1$
and  
	\begin{equation*} 
	H_n(x) =\delta  H_{n-1}(x) = \bigl (\delta^n 1 \bigr) (x), \quad n \ge 1.
	\end{equation*}
    In particular, $H_1(x)=x$, $H_2(x)=x^2-1$, $H_3(x)=x^3-3x$ and $H_4(x)=x^4-6x^2+3$.
\end{mydef}

\begin{prop}\label{prop:Hermite_Polynomials_Properties}
		$H_0(x),\dots, H_n(x)$ are monic polynomials spanning $\R_n[x]$ (the ring of polynomials of maximum degree $n$).
        
\end{prop}

Finally, we introduce the so-called gamma operator, which is widely used in applications of the Malliavin-Stein method. We shall use this operator in our study of the algebraic structure of PSO($H_p(N)$ in Section \ref{sec3.3}.

\begin{mydef} \label{def:Gamma_Y-dimension1}
	For a given (target) polynomial $y=h(x)$,  with $h\in \R[x]$, we introduce the linear operator $\Gamma_y$ acting on $\R[x]$ as 
	\begin{align*}
		\Gamma_y( f(x) ) = D h(x)  \delta^{-1}(f(x) ) =  \partial h(x) \delta^{-1}(f(x) ).
	\end{align*}
\end{mydef}
\begin{rem}{ \rm
		Let $X \sim N(0,1)$, and denote $Y=h(X)$ where $h$ is as above. Assume that $f,g \in \R[x]$.	By using integration by parts together with the classical chain rule we obtain
		\begin{align*}&
			\E\bigl[ g( Y) f(X) \bigr] -\E [ g(Y)] \E [f(X) ] = \E\bigl[ g( h(X) ) \delta \delta^{-1}(  f(X) ) \bigr]  \\ &
			= \E\bigl[  \partial (g \circ h) (X)  \delta^{-1} f(X) \bigr] = \E\bigl[  \partial g(Y) \Gamma_Y( f(X) ) \bigr]. 
		\end{align*}
	}
\end{rem}

\section{Algebraic structure of classes of polynomial Stein operators}\label{sec:Aplgebraic_Structures}


\subsection{General results and a connection between Stein's method and holonomic function theory}\label{sec3.1}

We open this section with a general result on the algebraic structure of $\PSO(X)$, the class of polynomial Stein operators of a given continuous target random variable $X$.  Parts (a) and (b) make a connection between Stein's method and holonomic function theory, and provide simple conditions under which the class $\PSO(X)$ is non-trivial, that is $\mathrm{PSO}(X) \neq \{0\}$. In the forthcoming Remark \ref{ex:(non)Holonomic_CF}, Proposition \ref{holprop} and Corollary \ref{cor3.1}, we will apply Proposition \ref{prop:PSO_General_Targets} in some concrete settings. Parts (c) of Proposition \ref{prop:PSO_General_Targets} tells us that for a given continuous target random variable $X$, the class $\PSO(X)$ can be generated by at most two polynomial Stein operators for $X$, whilst part (d) gives an initial simple result on the question of whether $\PSO(X)$ is principal. We consider the matter of whether $\PSO(X)$ is principal to be important, and will consider this questions in the concrete settings that $X\sim N(0,1)$ in Section \ref{sec3.2} and $X= H_p(N)$ in Section \ref{sec3.3}.

\begin{prop}\label{prop:PSO_General_Targets}
	\begin{itemize}
\item[(a)] Let $\varphi_X$ denote the characteristic function of the random variable $X$. Then $\mathrm{Ann}_{A_1(\mathbb{C})} (\varphi_X) \neq \{0\}$ if and only if $\varphi_X$ is holonomic. In addition, $\mathrm{PSO}(X) \neq \{0\}$ if and only if $\mathrm{Ann}_{A_1(\mathbb{C})} (\varphi_X) \neq \{0\}$.
 \item[(b)] 
 Let $X$ be an absolutely continuous random variable with probability density function given by $p_X(x)=\frac{\mathrm{d}F_X(x)}{\mathrm{d}x}$, where $F_X$ denotes the cumulative distribution function of $X$.
 Then, $\mathrm{PSO}(X) \neq \{0\}$ if and only if $p_X$ is holonomic. 

\item[(c)] $\mathrm{PSO}(X) \lhd_r  A_1(\R)$ is a \textbf{right} ideal of the first Weyl algebra $A_1(\R)$. Moreover, $\mathrm{PSO}(X)$ can be generated by two elements, i.e. $\mathrm{PSO}(X)= \langle  G_1, G_2 \rangle_{A_1(\R)}$ for some $G_1, G_2 \in A_1(\R)$.
\item[(d)] $\mathrm{PSO}(X)$ is a right principal ideal if and  only if $\mathrm{Ann}_{A_1(\mathbb{C})} (\varphi_X) \lhd_l  A_1(\mathbb{C})$ is a \textbf{left} principal ideal in the first Weyl algebra $A_1(\mathbb{C})$.
\end{itemize}
\end{prop}

We will need the following lemma in our proof of Proposition \ref{prop:PSO_General_Targets}.

\begin{lem}\label{lem:Anti-Isomorphism}
The Fourier  operator   $\Psi : A_1(\CC)  \to A_1(\CC)$, given by
\begin{align*} 
x^n \partial_x^k \mapsto 
\Psi\bigl( x^n \partial_x^k )=
\Ci^{k-n} t^k \partial_t^n,  \quad k,n\in \N_0,
\end{align*} 
extended by linearity is an anti-isomorphism of the Weyl algebra $A_1(\CC)$, that is to say, a bijection so that for every $S_1,S_2 \in A_1(\CC)$,
\begin{align*}
\Psi( S_1 S_2 ) =\Psi( S_2 ) \Psi(S_1).
\end{align*}
\end{lem}

\begin{proof}
Clearly, $\Psi$ is a bijective map. On the other hand, the set $\{   x^n \partial^k_x \, :\,  n,k \in \N_0\}$ provides a basis (as a vector space) over the field $\CC$; see \cite[Proposition 1.2]{Bjork-Book}. Therefore, since
\begin{align*}
\Psi\bigl( x^n \partial^k_x\bigr)
\Psi\bigl( x^m \partial^{\ell}_x\bigr)
= \Ci^{k+\ell-n-m} t^k \partial^n_t t^{\ell} 
\partial^m_t
= \Ci^{k+\ell-n-m} \sum_{r=0}^{n\wedge \ell}
\frac{n! \ell! }{r! (n-r)!(\ell-r)!}
t^{k+\ell-r} \partial_t^{m+n-r}
\end{align*}
and
\begin{align*} 
\Psi\bigl( x^m \partial^{\ell}_x x^n \partial^k_x \bigr)
&=\Psi\biggl(
\sum_{r=0}^{\ell\wedge n} \frac{ \ell!}{r! (\ell-r)!}
\frac{n!}{(n-r)!} x^{n+m-r}\partial_x^{k+\ell-r}
\biggr)\\
& = \sum_{r=0}^{\ell\wedge n} \frac{ \ell!n!}{r! (\ell-r)!(n-r)!} 
\Psi\bigl( x^{n+m-r}\partial_x^{k+\ell-r} \bigr)\\
&=  \Ci^{k+\ell-n-m} \sum_{r=0}^{n\wedge \ell}
\frac{ \ell!n! }{r! (\ell-r)!(n-r)!} 
t^{k+\ell-r} \partial_t^{m+n-r},
&\end{align*}
the result follows.
\end{proof}

\noindent{\emph{Proof of Proposition \ref{prop:PSO_General_Targets}.}}	(a) The first claim is a direct application of Definition \ref{def:Holonomic-Function}. For the second claim, first assume $0 \neq S \in \PSO(X)$.  Let $\mathcal{L} = \Psi(S)$ be the image of $S$ under the Fourier mapping $\Psi$ in Lemma \ref{lem:Anti-Isomorphism}. Then $ 0 \neq \mathcal{L}   \in \Ann_{A_1(\CC)} (\varphi_X)$ because $\mathcal{L}  \varphi_X (t) = \E[S e^{\Ci t X}] =0$ for every real $t$, where we applied the Stein operator $S$ to the real and imaginary parts of $e^{\Ci t X}$. 
For the other direction, 
assume that $\Ann_{A_1(\CC)}(\varphi_X)$ is non-trivial. In other words, the characteristic function $\varphi_X$ is holonomic. Consider the algebraic function $g(t)=-t$ for every real $t$. Then Theorem \ref{thm:Holonomic-Closure-Properties} implies that the function $\varphi_X \circ g$ is holonomic. Take $0 \neq  \mathcal{L}  \in \Ann_{A_1(\CC)} (\varphi_X \circ g)$ and let $S = \Psi^{-1}(L)$.  Then $0 \neq S \in A_1(\R)$, and $\E[S e^{-\Ci tX}]= \mathcal{L}  (\varphi_X \circ g)(t) =0$ for every real $t$. In addition, for an arbitrary $f \in \mathcal{S}(\R)$, the Schwartz class, by Fourier inversion formula, we can write 
		\begin{align*} f(x) = \frac{1}{2\pi} \int_{\R}  e^{-\Ci t x} \widehat{f}(t) \,dt. 
		\end{align*}
		Therefore, 
		\begin{align*}
		\E[S f(X)] &= \E  \left[  S \biggl( \frac{1}{2\pi} \int_{\R}  e^{-\Ci t X} \widehat{f}(t)\, dt   \biggr)\right] = \frac{1}{2\pi}\E   \left[   \int_{\R}  S (e^{-\Ci t X} )  \widehat{f}(t)\, dt \right] \\ 
		& =\frac{1}{2\pi} \int_{\R} \widehat{f}(t) \E  \big[   S (e^{-\Ci t X} ) \big]  \,dt 
		 = \frac{1}{2\pi} \int_{\R} \widehat{f}(t)  \mathcal{L}  (\varphi_X \circ g)(t) \, dt =0,
		\end{align*}
        where we used Fubini's theorem in obtaining the third equality.
	Note that in obtaining the second equality, we may interchange the integral and the operator $S$ because  $\widehat{f}$ is also in $\mathcal{S}(\R)$ (we emphasis that the operator $S$ acts on the variable $x$ and not on $t$).  Hence, $S\in \PSO(X)$.
	
\vspace{2mm}
	
\noindent{(b)}  Suppose that $X$ is a continuous random variable with holonomic density $p_X$ that is supported on the interval $(a,b)\subseteq\mathbb{R}$, where $-\infty\leq a<b\leq \infty$. Since the density $p_X$ of $X$ is holonomic, we know from part (b) in Theorem \ref{thm:Holonomic-Closure-Properties} that $p_X$ has only finitely many singularities. 
These singularities can occur at the points $a$ and $b$ or in the interior of the interval $(a,b)$, which we denote by $a<x_1<x_2<\ldots<x_{n-1}<b$, and we also denote $x_0=a$ and $x_n=b$. Since $p_X$ is holonomic, it follows that there exists $L_i\in A_1(\mathbb{R})\setminus\{0\}$ such that $L_ip_X(x)=0$ on the interval $(x_{i-1},x_i)$, $i=1,\ldots,n$. We now consider the asympototic behaviour of the density $p_X$ in small neighbourhoods of the singularities. Since $\int_a^b p_X(x)\,dx=1<\infty$, it follows from Theorem \ref{thmsing} that, for $0\leq i\leq n$,
\begin{equation*}
p_X(x)    \sim  (x-x_i)^{\alpha_i} \, \sum_{j=0}^{\infty}  Q_{i,j} (\log (x-x_i)) (x-x_i)^{j\beta_i},    \quad  x \to  x_i,
\end{equation*}
where $\alpha_i>-1$, $\beta_i \in \mathbb{Q}_{>0}$ and the $Q_{i,j}$ are a family of polynomials of uniformly bounded degree. 

We will make use of the fact that the ring of linear differential operators with polynomial coefficients is a left Euclidean domain \cite{Ore}. To this end, we let $L\not=0$ denote the least common left divisor of $L_1,\ldots L_k$. Now, clearly the operator $R=\prod_{i=0}^{n} (x-x_i)^{\mathrm{ord}(L)} L$ annihilates the entire density function $p_X$ where $\mathrm{ord}(L)$ is the order of the operator $L$ (it is known that $\mathrm{ord}(L) \le \sum_i  \mathrm{ord}(L_i)$, see \cite[Theorem 6]{BoundOnOrder}). Moreover, the operator $S=\prod_{i=0}^{n} (x-x_i)^{\alpha+1+\mathrm{ord}(L)}$, where $\alpha=\max_i\alpha_i$, also annihilates the entire density function $p_X$ over the interval $(a,b)$. We also note that since $S\in A_1(\mathbb{R})$, it follows that $S$ has an adjoint $S^*$ (as given by (\ref{eq:Algebraic-Adjoint})) which is also an element of $A_1(\mathbb{R})$.

Since $Sp_X(x)=0$ for all $x\in\mathbb{R}$, it follows that for $f\in C_c^\infty(\mathbb{R})$ we have that $\int_a^bf(x)Sp_X(x)\,dx=0$. Therefore, by Green's formula (\ref{eq:Green-Formula}),
\begin{align}
0&=\int_a^bf(x)Sp_X(x)\,dx=
\sum_{i=1}^n\int_{x_{i-1}}^{x_i}f(x)Sp_X(x)\,dx \nonumber\\
&=\sum_{i=1}^n\bigg\{\int_{x_{i-1}}^{x_i}p_X(x)S^*f_X(x)\,dx+\Big[P_S(p_X(x),f(x))\,\Big]_{x_{i-1}}^{x_i} \bigg\}\nonumber\\
&=\sum_{i=1}^n\int_{x_{i-1}}^{x_i}p_X(x)S^*f_X(x)\,dx=\int_{a}^{b}p_X(x)S^*f_X(x)\,dx=\mathbb{E}[S^*f(X)], \label{gf}
\end{align}
where in the third step we used that $P_S(p_X(x_i),f(x_i))=0$ for all $0\leq i\leq n$ (where the bilinear concomitant $P_S$ is defined as in (\ref{eq:Bilinear-Concomitant})), which is a consequence of our construction of the operator $S$. Thus, $S^*\in \mathrm{PSO}(X)$, and we conclude that $\mathrm{PSO}(X)\not=\{0\}$.

Suppose now that $\PSO(X) \neq \{0\}$. By part (a), the assumption $\PSO(X) \neq \{0\}$ implies that the characteristic function $\varphi_X$ is holonomic. Now, by the Gil-Pelaez inversion formula \cite{g51} we have that, for $x\in\mathbb{R}$,
\begin{align}\label{levy}
F_X(x)=\frac{1}{2}-\frac{1}{\pi}\int_{0}^\infty \frac{\mathrm{Im}[e^{-\mathrm{i}tx}\varphi_X(t)]}{t}\,dt,   
\end{align}
with the formula applying since $x$ is a continuity point of $F_X$ as $X$ is an absolutely continuous random variable. The integrand $\mathrm{Im}[e^{-\mathrm{i}tx}\varphi_X(t)]/t$ is holonomic because $\varphi_X$ is holonomic and the product of two holonomic functions is holonomic (by part (a) of Theorem \ref{thm:Holonomic-Closure-Properties}), and the imaginary part of a holonomic function is also holonomic. Holonomicity is also preserved under integration (again, see part (a) of Theorem \ref{thm:Holonomic-Closure-Properties}) and therefore by (\ref{levy}) it follows that the cumulative distribution function $F_X$ is holonomic. 
By assumption, $F_X$ is differentiable with derivative $p_X$. Since holonomicity is preserved under differentiation (part (a) of Theorem \ref{thm:Holonomic-Closure-Properties}), it follows that the density $p_X$ is holonomic.
    
	
\vspace{2mm}
	
\noindent{(c)} Clearly $(\PSO(X), +)$ is a subgroup of $(A_1(\R),+)$. Now the claim follows at once due to the fact that the class $\mathcal{F}$ is invariant under elements of the first Weyl algebra $A_1(\R)$. Finally, that $\PSO(X)$ can be generated by two elements is Stafford's Theorem \ref{thm:Stafford}. 
	   
\vspace{2mm}
	   
\noindent{(d)} It is easily seen that $\Ann_{A_1(\mathbb{C})} (\varphi_X) \lhd_l A_1 (\CC)$ is a left ideal of the first Weyl algebra $A_1(\CC)$. Moreover, the anti-isomorphism $\Psi$ given in Lemma \ref{lem:Anti-Isomorphism} sends back and forth a generator to a generator. \hfill $\Box$

\begin{rem}\label{ex:(non)Holonomic_CF}{ \rm
(a) Many continuous probability distributions admit holonomic characteristic functions, including the Gaussian, exponential, gamma, beta, 
variance-gamma,
Pearson distributions,  and semicircular just to mention a few. 
Therefore, according to Proposition \ref{prop:PSO_General_Targets} part (a), the associated polynomial Stein operator classes are non-trivial; indeed, polynomial Stein operators are known for all these distributions as well as many others; we refer the reader to \cite{a-g-g-algebraic-stein-operators,gms19,ley} and references therein for an overview of this literature.



\vspace{2mm}

\noindent{(b)} Let $X$ follow the extreme value Gumbel distribution with density function $p_X(x)= e^{-x} e^{-e^{-x}}$, $x \in \R$. All moments of the Gumbel distribution exist. It is known that the characteristic function is given by $\varphi_X(t) =\Gamma(1 - \Ci t)$, where the gamma function is defined by $\Gamma(z)=\int_{0}^{\infty} x^{z-1} e^{-x}\,dx$, for $\text{Re}(z) >0$. The gamma function $\Gamma(z)$ is not a holonomic function, see, e.g.,  \cite[p.\ 4]{Holonomic-Toolkit}, and therefore the characteristic function $\varphi_X$ is not holonomic. Therefore, $\PSO(X) =0$ is the trivial ideal. For Gumbel approximation via Stein's method the reader can consult \cite{ley, Gumber_Stein}.

\vspace{2mm}

\noindent{(c)} Let $X$ follow the log-normal distribution with density function given by $p_X(x)= (2\pi)^{-1/2}x^{-1}\exp(-(\log(x))^2/2)$, $x>0$. All moments of the log-normal distribution exist. It is readily seen that $p_X(x)$ is not a holonomic function \cite{Non-Holonomic-Character},
and therefore $\mathrm{PSO}(X)=\{0\}$ is the trivial ideal. A brief account of the Malliavin-Stein method for log-normal approximation is given by \cite{kusuotud}. We also remark that it is a well-known open problem to find a closed-form expression for the characteristic function of the log-normal distribution. Proposition \ref{prop:PSO_General_Targets} tells us that the characteristic function of the log-normal distribution is not holonomic and thus does not satisfy a homogeneous, linear ordinary differential equation with rational coefficients. This partially explains the difficulty in finding a closed-form formula for the characteristic function, because many of the standard special functions are solutions of homogeneous, linear ODEs with rational coefficients.
}
\end{rem}

From Proposition \ref{prop:PSO_General_Targets} and the closure properties of holonomic functions, we deduce the following proposition which shows that the class of polynomial Stein operators for a wide family of distributions is non-trivial. 

\begin{prop}\label{holprop}Suppose $h\in\mathbb{R}[x]$ (the ring of polynomials with real-valued coefficients) and let $X$ be a real-valued random variable with holonomic density. Let $Y=h(X)$. Then $\mathrm{PSO}(Y)\not=\{0\}$.
\end{prop}

As the Gaussian density is holonomic, the following corollary is immediate. This answers one of the open problems raised by our recent work \cite{a-g-g-algebraic-stein-operators}.

\begin{cor}\label{cor3.1} Let $N\sim N(0,1)$.  Then, for any $h\in\mathbb{R}[x]$, we have  $\mathrm{PSO}(h(N))\not=\{0\}$.
\end{cor}

\noindent{\emph{Proof of Proposition \ref{holprop}.}} Suppose without loss of generality that $h\in\mathbb{R}[x]$ has  $k\geq0$ distinct local maxima and minima, which we label by $y_1<y_2<\cdots<y_k$. Suppose that the support of $Y=h(X)$
 is given by the interval $(a,b)\in\mathbb{R}$, and for ease of notation define $y_0:=a$ and $y_{k+1}:=b$. Suppose $y\in(y_i,y_{i+1})$ for $0\leq i\leq k$.  Then, the density of the random variable $Y=h(X)$ can be expressed as
 \begin{align*}f_Y(y)=\sum_{h(x)=y}\frac{f_X(x)}{|h'(x)|}.
 \end{align*}
From the closure properties of holonomic functions (see Theorem \ref{thm:Holonomic-Closure-Properties}), we may now argue that the density of $Y$ is holonomic. We have that $y\mapsto x_y$, a solution of $h(x_y)=y$, is holonomic, as the equation is algebraic. Also, $g=1/h'$ is holonomic, as it satisfies the ODE $(h'(x))^2g'(x)+h''(x)=0$. Thus,  by composition, $y\mapsto1/h'(x_y)$ is holonomic. By the same reasoning, $y\mapsto f_X(x_y)$ is holonomic (by assumption $f_X$ is holonomic), and by product $f_X(x_y)/|h'(x_y)|$ is holonomic. Thus, we have shown that the density $f_Y$ is holonomic on each of the intervals $(y_i,y_{i+1})$, $0\leq i\leq k$. 

To conclude that $f_Y$ is holonomic over its whole support, we make use of a classical fact that the ring of linear differential operators with polynomial coefficients is a left Euclidean
domain \cite{Ore}. By the discussion above, we can write $f_Y(y) = \sum_{i=0}^{k} f_i (y) \textbf{1}_{\{ (y_i,y_{i+1})\} }(y)$, where each $f_i$ is holonomic. Let $0\le i \le k$ and assume that $ 0 \neq L_i \in A_1 (\R)$ annihilates the density function $f_Y$ over the interval $(y_i,y_{i+1})$. Denote by $L$ the least common left divisor of $L_1,\ldots,L_k$. Note that $L \neq 0$ and clearly the operator $S=\prod_{i=0}^{k} (y-y_i)^{\mathrm{ord}(L)} L$ annihilates the entire density function $f_Y$ where here $\mathrm{ord}(L)$ stands for the order of the operator $L$ (it is known that $\mathrm{ord}(L) \le \sum_i  \mathrm{ord}(L_i)$, see \cite[Theorem 6]{BoundOnOrder}). The result now follows by part (b) of Proposition \ref{prop:PSO_General_Targets}.
\hfill $\Box$

\subsection{Algebraic structure of PSO($N$) }\label{sec3.2}

Our description of the algebraic structure of $\PSO(N)$ is given in the following theorem.

\begin{thm}\label{prop:PSO=Ideal}
Consider $\mathrm{PSO}(N)$, the class of all polynomial Stein operators associated to the standard Gaussian random variable $N$.  Then the following statements are in order:
\begin{itemize}
\item[(a)] $\mathrm{PSO}(N)$ is a vector space over $\R$ with basis
\begin{align}\label{eq:PSO-basis}
\mathbf{\mathfrak{B}}:=  \Big\{    S(k,t):= H_k\partial^t- H_{k+t},\quad k,t\in \N_0 \Big\}.
\end{align}
\item[(b)] $\mathrm{PSO}(N)$ is a principal maximal right ideal, and
\begin{equation}\label{eq:Generator}
\mathrm{PSO}(N) = \langle G \rangle _{A_1(\R)}= G A_1(\R) : = \big\{    G L \, :  \, L \in A_1(\R)       \big \},
\end{equation}
 where the generator $G=\partial-x$. (Equivalently, the ideal $\mathrm{PSO}(N)$, as a module on the first Weyl algebra $A_1(\R)$, is a cyclic module.) 
\item[(c)] The ideal $\mathrm{PSO}(N)$ (as a $A_1(\R)$ module) is \textbf{not} holonomic, although the quotient $A_1(\R) / \mathrm{PSO}(N)$ is a holonomic module.  In addition, $d (  A_1(\R) / \mathrm{PSO}(N))=1$ and $d(\mathrm{PSO}(N))=2$.
\end{itemize}
	\end{thm}
\begin{rem}{ \rm
\begin{itemize}
\item[(a)] The first Weyl algebra $A_1(\R)$ is not a principal ideal domain. In fact, the ideal 
generated by elements $x^2$ and $1+x\partial$ is not principal, see \cite[Example 7.11.8]{mr-Book87}. 
\item[(b)]Items (b)-(c) of Theorem \ref{prop:PSO=Ideal} reveal that $\PSO(N)$ is a cyclic non-holonomic module on the first Weyl algebra $A_1(\R)$. It is a well-known fact in $D$-module theory that every holonomic module is cyclic; see  Theorem \ref{thm:Stafford} item (b).
\end{itemize} 
}
\end{rem}

\begin{rem}\label{rem:PSO(N)-Generators}
{\rm  Consider the right ideal $I=\langle  G  \rangle \lhd_r A_1(\R)$, where $G=  \partial-x$. Clearly, $I \subseteq \PSO(N)$. The quotient ideal $A_1(\R)/I$ is a holonomic module (see e.g. \cite[p.\ 86]{Coutinho-Book}) and by part (b) of Stafford's Theorem  \ref{thm:Stafford} is cyclic meaning that there is $G' \in A_1(\R)$ so that $A_1(\R)/I  =  \langle  G' + I \rangle$. Therefore, $\PSO(N) = \langle G, G' \rangle_{A_1(\R)}$. However, Theorem \ref{prop:PSO=Ideal} part (b) yields that the extra generator $G'$ is immaterial. 	
}
\end{rem}

In proving part (b) of Theorem \ref{prop:PSO=Ideal} we will make use of the following simple lemma.

\begin{lem}\label{lem:delta-operate-on-basis}
Let $G= \partial-x \in A_1(\R)$. Then, for every $k,l\in \N_0$, it holds that

\begin{equation*}
  G ( H_k \partial^l ) =  H_k \partial^{l+1}-H_{k+1} \partial^l.
  \end{equation*}
\end{lem}

\begin{proof}
By definition of the operator $G$, we can write 
\begin{align*}
G  ( H_k \partial^l )  & =  \partial ( H_k \partial^l )-x H_k \partial^l=  \partial (H_k) \partial^l + H_k \partial^{l+1}-x H_k \partial^l\\
& =   H_k \partial^{l+1}-\left(  xH_k - \partial H_k \right) \partial^l
=H_k \partial^{l+1}- \delta(H_k) \partial^l \\
&=   H_k \partial^{l+1}-H_{k+1} \partial^l,
\end{align*}
where we used Definition \ref{def:Hermite_Polynomials} of Hermite polynomials.
\end{proof}

\begin{proof}[Proof of Theorem \ref{prop:PSO=Ideal}]
	(a) It is clear that $\PSO(N)$ is a linear space. Now  let $S=\sum_{t=0}^{T} p_t \partial^t \in A_1(\R)$. Observe that by using the Gaussian integration by parts formula \eqref{lem:GIBP-d=1}, $S \in \PSO(N)$ if and only if 
	\begin{equation*}
	p_0=-\sum_{t=1}^{T} \delta^t p_t.
	\end{equation*}
On the other hand, every polynomial can be written as a finite sum of Hermite polynomials; see Proposition \ref{prop:Hermite_Polynomials_Properties}. Hence, a basis of the linear space of Stein operators for the standard Gaussian distribution is given by the elements 
\[  S(k,t)= H_k \partial^t - \delta^t H_k = H_k \partial^t - H_{k+t},  \quad t, k \in \N_0.  \] 
(b)  In order to show that $\PSO(N)$ is a principal ideal, it is enough to show that for every element $S(k,t) \in \mathfrak{B}$ we have $S(k,t) = G L$ for some element $L \in A_1(\R)$. 	By taking  telescoping sums and applying Lemma \ref{lem:delta-operate-on-basis} in the second step,
\begin{align*} 
S(k, t) &=- \sum_{r=1}^t 
\biggl(H_{k+t+1-r} \partial^{r-1} -H_{k+t-r} \partial^r
\biggr)=G  \biggl(  \sum_{r=1}^ t H_{k+t+1-r} \partial^{r-1} \biggr). 
\end{align*}
Maximality of $\PSO(N)$ follows from \cite[Exercise 18, p.\ 50]{Lam-Book}. Maximality of principal ideals $xA_1(\R)$ and $A_1(\R)x$ is also discussed in \cite[p.\ 21]{mr-Book87}. 

\vspace{1mm}

\noindent{(c)} According to \cite{Coutinho-Book}, Corollary 1.2 and Proposition 1.3, Chapter 10, torsion and holonomic finitely generated $A_1(\R)$ modules are the same thing. Now, if $\PSO(N)$ were a holonomic module then this would imply that the generator $G=\partial-x$ would be a torsion element, i.e., $G L =0$ for some non-zero element $L \in A_1(\R)$, which is a contradiction with $A_1(\R)$ being an integral domain. Lastly, the claim $d(\PSO(N))=2$ follows directly from the Bernstein's inequality.
	\end{proof}

\subsection{Algebraic structure of PSO($H_p(N)$)}\label{sec3.3}

We begin this section by proving that the ideal $\PSO(H_2(N))$ is principle, where $N\sim N(0,1)$. In Proposition \ref{prop_h2_gen} below we shall then prove that the operator $G=2(y+1)\partial-y\in\PSO(H_2(N))$ is the generator of the ideal.

\begin{ex}\label{ex:H_2(X)IsPrincipal}
{  \rm Let $Y=H_2(N) = N^2-1$, where $N \sim N(0,1)$. In this example, we show that the ideal $\mathrm{PSO}(Y)$ is principal. Recall that the operator $G = 2 (y+1)\partial - y \in \mathrm{PSO}(Y)$. Let $L = \sum_{t=0}^{T} p_y(y) \partial^t \in \mathrm{PSO}(Y)$ be arbitrary, where $T\ge2$. According to \cite[p.\ 19]{Coutinho-Book}, we need to show that there exists a polynomial $p_0(y) \in \R[y]$ so that $p_T(y)\partial + p_0(y) \in \mathrm{PSO}(Y)$. In order to encode the existence of the polynomial $p_0$, note that a simple application of the Gaussian integration by parts formula yields that 
\begin{equation*}
\E \left[  p_T (Y) \partial f(Y)   \right] = \E \left[           \delta \left( \frac{p_T(Y)}{2X} \right) f(Y)    \right].
\end{equation*}	
Denote $p_0 (y) : = \delta (\frac{p_T (y)}{2x})$, where here we have the association $y = H_2 (x) = x^2 -1$. Hence, we are left to show that $\delta (\frac{p_T(y)}{x}) \in \R [y]$.  Now, from Proposition 4.1 \cite{a-g-g-algebraic-stein-operators}, one can deduce that $p_T(y) = (y+1)h(y) = x t (x)$ for some polynomials $h(y) \in \R[y]$, and $t(x) \in \R[x]$ (observe that $p_T(y)/x \in \R[x]$). Next, we have 
\begin{equation*}
	\delta\bigg(\frac{p_T(y)}{x}\bigg) = \left(  x - \partial \right)\left(  \frac{p_T(y)}{x}  \right) = p_T(y) - \partial_x \left(   \frac{(y+1)h(y)}{x}\right).
\end{equation*}	  
Moreover,
\begin{align*}
\partial_x \left(  \frac{(y+1)h(y)}{x} \right) &= \frac{(2x h(y) + 2x (y+1)h'(y))2x - 2 (y+1) h(y)}{4x^2}\\
&=\frac{1}{2} h(y)+(y+1)h'(y) \in \R[y],
\end{align*}	
and hence we are done. 
}

\begin{lem}\label{lem:TowardsGenerator}
Let $Y=H_2(N)$, where $N\sim N(0,1)$. Consider $L= p_0 (y) + p_1(y) \partial$ with $p_0, p_1 \in \R[y]$, where we have identified $y=H_2(x)=x^2-1$. Then $$ L \in \mathrm{PSO}(Y) \quad \text{ if and only if } \quad L \in\langle G \rangle_{A_1(\R)},$$ where $G=-y + 2(y+1)\partial$. 
\end{lem}

\begin{proof}
It is well-known that $G \in \mathrm{PSO}(Y)$ (see, for example, \cite{np09}), and so it is obvious that when $ L \in \langle G \rangle_{A_1(\R)}$ then $L \in \mathrm{PSO}(Y)$. For the other side, assume that  $L= p_0 (y) + p_1(y) \partial \in \mathrm{PSO}(Y)$. Now, from Proposition 4.1 \cite{a-g-g-algebraic-stein-operators}, one deduces that $p_1(y) = 2(y+1)h(y)$ for some polynomials $h(y) \in \R[y]$. Since $L$ is algebraic,  
\begin{align*}
p_0 (y) = - \delta \left(  \frac{(y+1)h(y)}{x} \right)
= -y h(y) + 2 (y+1)h'(y).
\end{align*}
Hence, 
\begin{align*}
L &= p_0 (y) + p_1(y) \partial = \left( -y h(y) + 2 (y+1)h'(y) \right) + 2(y+1)h(y) \partial \\
& = G h(y) \in \langle G \rangle_{A_1(\R)}.
\end{align*}
This completes the proof of the lemma.
\end{proof}

\begin{prop}\label{prop_h2_gen}
Let $Y=H_2(N)$, where $N\sim N(0,1)$. Then, $\mathrm{PSO}(Y) = \langle G \rangle_{A_1(\R)}$,  where the generator $G=-y + 2(y+1)\partial$.
\end{prop}

\begin{proof}
Let $L = \sum_{t=0}^{T} p_t (y)\partial^t \in \mathrm{PSO}(Y)$. Since we know that $\mathrm{PSO}(Y)$ is principal, it follows that there exists a polynomial $\widetilde{p}_{0,T} (y) \in \R[y]$ so that the operator $\widetilde{L}_T := \widetilde{p}_{0,T} (y) + p_T (y) \partial \in \mathrm{PSO}(Y)$.  Now, by Lemma \ref{lem:TowardsGenerator}, we have $\widetilde{L}_T \in \langle G \rangle_{A_1(\R)}$. Note that
\begin{align*}
L - \widetilde{L}_T \partial^{T-1} = \left( p_{T-1} (y) - \widetilde{p}_{0,T} (y) \right) \partial^{T-1} + \sum_{t=0}^{T-2} p_t (y) \partial^t = :S_{T-1}.
\end{align*}
This implies that $L \in \langle G \rangle_{A_1(\R)}$ if and only if $S_{T-1} \in \langle G \rangle_{A_1(\R)}$. Next, observe that $S_{T-1} \in \mathrm{PSO}(Y)$ due to the fact that $\mathrm{PSO}(Y)$ is a right ideal. Denote $\widetilde{q}_{T-1}(y) = p_{T-1} (y) - \widetilde{p}_{0,T} (y)$. In a similar way, there exists a polynomial $\widetilde{p}_{0,T-1}$ so that the operator $$\widetilde{L}_{T-1} = \widetilde{p}_{0,T-1} (y) + \widetilde{q}_{T-1} (y)\partial \in \langle G \rangle_{A_1(\R)}.$$
Note that $ S_{T-1}  - \widetilde{L}_{T-1} \partial^{T-2} = \left(  p_{T-2} (y) -  \widetilde{p}_{0,T-1} (y) \right) \partial^{T-2} + \sum_{t=0}^{T-3} p_t (y) \partial^t = : S_{T-2}$. Note that, again, $ S_{T-1} \in \langle G \rangle_{A_1(\R)}$ if and only if $S_{T-2} \in \langle G \rangle_{A_1(\R)}$. By repeating this argument, we deduce that $L \in \langle G \rangle_{A_1(\R)}$ if and only if for some operator $S_1 =\left(p_1(y)- \widetilde{q}_1 (y) \right) \partial+ p_0 (y) \in \mathrm{PSO}(Y)$ we have that $S_1 \in \langle G \rangle_{A_1(\R)}$. Now one can complete the proof by a direct application of Lemma \ref{lem:TowardsGenerator}.
\end{proof}
\end{ex}	

In the following examples, we show that the ideal $\PSO(H_p(N))$ is not principle for $p=3,4$.

\begin{ex}\label{ex:H_3(X)IsNotPrincipal} {  \rm
Let $Y=H_3(N) = N^3-3N$, where $N \sim N(0,1)$. In this example, we show that the ideal $\mathrm{PSO}(Y)$ is not principal. Appendix A of \cite{a-g-g-algebraic-stein-operators} reveals that the ideal $\mathrm{PSO}(Y)$ contain a fifth order operator with leading polynomial coefficient $p_5 (y) = y^2 -4$, and an operator of order three. So, following \cite[page 19]{Coutinho-Book}, if the ideal $\mathrm{PSO}(Y)$  was principal, then there would exist polynomials $p_0(y), p_1(y), p_2(y) \in \R[y]$ so that for the operator 
\begin{equation*}
	(y^2-4) \partial^3 + p_2 (y)\partial^2 + p_1(y) \partial + p_0 (y) \in \mathrm{PSO}(Y).
	\end{equation*}
The rest is to show that this is impossible. By contradiction, if this is the case, then Proposition 3.2 of \cite{a-g-g-algebraic-stein-operators} yields that (note that $y = H_3 (x) = x^3-3x$)
\begin{align*}
 \Gamma^3_y (p_0(y)) + \Gamma^2_y (p_1(y)) + \Gamma_y (p_2(y))   &= \Gamma_y (   \Gamma^2_y (p_0(y)) + \Gamma_y (p_1(y)) + p_2(y)     ) \\
 &=3 (x^2-1) \delta^{-1} (g(x))\\
&= -(y^2-4) = -(x^2-4)(x^2-1)^2.
\end{align*}	
This implies that the polynomial $g(x) :=  \Gamma^2_y (p_0(y)) + \Gamma_y (p_1(y)) + p_2(y)$ must be of degree five. Next, we proceed with a degree type argument on the variable $x$. First, note that the degree of the polynomial $p_2(y)$ (w.r.t the
variable $x$) is $3\deg(p_2)$. Similarly, $\deg (\Gamma_y (p_1(y))) = 3 \deg (p_1) +1$, and $\deg( \Gamma^2_y (p_0(y))) = 3 \deg(p_0)+2$. Observe that for every three integers $d_1, d_2, d_3 \in \N_0$ it is the case that $3d_1 +2 \neq 3d_2 +1, 3d_1+2 \neq 3d_3, 3d_2+1 \neq 3 d_3$. Hence, 
\begin{equation*}
\deg (g(x)) = \max \Big\{        3 \deg(p_0)+2,      3 \deg (p_1) +1,  3\deg(p_2) \Big \} =5.
\end{equation*}	 
The only possibility is that $\deg (p_0)=1$, $\deg(p_1) \le 1$ and $\deg(p_2) \le 1$. Now, we are left to show that there do not exist constants $c\neq 0$ and $a_1, b_1, a_2, b_2 \in \R$ so that 
\begin{equation*}
L:=c y + (a_1y +b_1) \partial + (a_2 y + b_2) \partial^2 + (y^2-4)\partial^3 \in \mathrm{PSO}(Y).
\end{equation*}	
This is straightforward by applying polynomials $f(y) =y^n$, $n \in \N_0$, in the requirement $\E \left[ L f (Y)\right]=0$ to obtain a contradiction. 
}
\end{ex}	

\begin{ex}\label{ex:H_4(X)IsNotPrincipal} {\rm
Let $Y=H_4(N) = N^4 -6N^2+3$, where $N\sim N(0,1)$. In this example, we show that the ideal $\mathrm{PSO}(Y)$ is not principal. Appendix A of \cite{a-g-g-algebraic-stein-operators} reveals that the ideal $\mathrm{PSO}(Y)$ contain a third order operator with leading polynomial coefficient $p_3 (y) = (y-3)(y+6)$, and an operator of order two. Now, again following \cite[page 19]{Coutinho-Book}, if the ideal $\mathrm{PSO}(Y)$  would be principal, then there would exist polynomials $p_0(y), p_1(y) \in \R[y]$ so that  
\begin{equation*}
	(y-3)(y+6) \partial^2 + p_1 (y)\partial+ p_0 (y) \in \mathrm{PSO}(Y).
\end{equation*}
The rest is to show that this is impossible. By contradiction, if this is the case, then Proposition 3.2 of \cite{a-g-g-algebraic-stein-operators} yields that (note that $y = H_4 (x) = x^4-6x^2+3$):
\begin{align*}
\Gamma^2_y (p_0(y)) + \Gamma_y (p_1(y)) &=	\Gamma_y ( \Gamma_y (p_0(y)) + p_1(y)  ) = H'_4(x) \delta^{-1} (\Gamma_y (p_0(y)) + p_1(y)) \\
&= -(y-3)(y+6) = - x^2 (x^2-3)^2 (x^2-6).
\end{align*}	
This implies that the polynomial $g(x): =\Gamma_y (p_0(y)) + p_1(y)$ is of degree six with respect to the variable $x$. On the other hand,  consider that the degree of the polynomial $p_1(y)$ (w.r.t the variable $x$) is $4 \deg(p_1)$, and, similarly, $\deg_x (\Gamma_y (p_0(y))) = 4 \deg(p_0)+2$.  Now, note that for every two non-negative integers $d_1, d_2 \in \N_0$, we have that $4d_1+2 \neq 4d_2$.  This yields that 
\begin{align*}
\deg (g(x)) = \max \Big \{      4 \deg (p_0)  +2, 4 \deg(p_1)            \Big \} =6.
\end{align*}	
This is the case only when $\deg(p_0) =1$ and $\deg(p_1) \le 1$. Now, we are left to show that there do not exist constants $c\neq 0$ and $a_1, b_1\in\mathbb{R}$ so that 
\begin{equation*}
L := cy + (a_1 y +b_1) \partial + (y-3)(y+6)\partial^2 \in \mathrm{PSO}(Y).
\end{equation*}	
But it is straightforward to see that the requirement $\E\left[ L f (Y) \right] =0$ with $f(y)=y^n$, and $n=0,1,\ldots,5$, yields that $c=0$, hence leading to a contradiction.}  
\end{ex}

\begin{rem}\label{remprinciple} {\rm Let $N\sim N(0,1)$. We conjecture that the ideal $\mathrm{PSO}(H_p(N))$ is principle if and only if $p=1,2$. Our method for proving that the ideals $\mathrm{PSO}(H_3(N))$ and $\PSO(H_4(N))$ are not principle could be applied to prove that the ideal $\mathrm{PSO}(H_p(N))$ is not principle for a given $p\geq5$ (assuming that this is the case), although the calculations would become increasingly complex as $p$ increases. 

The results of Examples \ref{ex:H_3(X)IsNotPrincipal} and \ref{ex:H_4(X)IsNotPrincipal} are of interest in the context of developing the Malliavin-Stein method for distributional approximations concerning limit distributions from the third and higher order Wiener chaoses, which is a significant open problem. Since $\mathrm{PSO}(H_3(N))$ and $\PSO(H_4(N))$ are not principle, it means that it is less obvious \emph{a priori} as to how one should determine which Stein operator(s) from the class $\mathrm{PSO}(H_p(N))$, $p\geq3$, are likely to be most well-suited for applications in the Malliavin-Stein method. Perhaps, on the flip side, there could be an added flexibility in which certain Stein operators from the class  $\mathrm{PSO}(H_p(N))$, $p\geq3$, are more suited to particular applications. In any case, the appreciation that the ideal $\mathrm{PSO}(H_3(N))$ is not principle is of potential value in guiding researchers in the development of the Malliavin-Stein method for limit distributions from third and higher order Wiener chaoses.
}   
\end{rem}

\begin{rem}\rm{
The principle property of PSO is not invariant under addition and multiplication. This can be seen from the following simple examples. Let $Y=H_3(N)$ and $X_1=N^3$ and $X_2=-3N$, so that $Y=X_1+X_2$. Then, whilst $\PSO(X_1)$ (this is verified by a similar argument to that given in Example \ref{ex:H_2(X)IsPrincipal}; we omit the details) and $\PSO(X_2)$ are principle, $\PSO(Y)$ is not principle. For multiplication, we take $Y=H_3(N)$ and $X_3=N$ and $X_4=N^2-3$, so that $Y=X_3 \times X_4$. Then, whilst $\PSO(X_3)$ and $\PSO(X_4)$ are principle, $\PSO(Y)$ is not principle.}
\end{rem}

\section{Characterising feature of \PSO($N$)}\label{sec4}

In this section, we study the characterising feature of $\PSO(N)$ for $N\sim N(0,1)$. It turns out to be convenient to study this problem by analysing solutions of differential equations that are satisfied by the Gaussian characteristic function. We open in Section \ref{sec4.1} by developing some general theory in this direction. We also provide some simple examples of non-characterising Stein operators for the standard Gaussian distribution.
In Section \ref{sec4.3}, we see that all first order polynomial Stein operators for the standard Gaussian distribution are characterising. We then move on to consider higher order polynomial Stein operators, for which the picture is much more complicated on account of the fact (as seen in Section \ref{sec4.1}) that there exist polynomial Stein operators for the standard Gaussian of order greater than one that are not characterising.




\subsection{The annihilator ideal structure}\label{sec4.1}

Denote by $\varphi_N (t)= e^{-t^2/2}$ the characteristic function of a standard Gaussian random variable $N$. Consider the left annihilator ideal $\Ann_{A_1(\mathbb{C})} (\varphi_N) : =  \big\{  \mathcal{S} \in A_1 (\mathbb{C}) \, :  \, \mathcal{S}\varphi_N =0  \big \}$.  In virtue of Theorem \ref{prop:PSO=Ideal} item (b), 
 $\Ann_{A_1(\mathbb{C})}(\varphi_N)$ is a left principal ideal generated by the operator $\mathcal{G}= \Psi (G)= \Psi (\partial -x) =  \partial +t$, where $\Psi$ is the Fourier map as introduced in Lemma \ref{lem:Anti-Isomorphism}, i.e.,  
\begin{equation*}
\Ann_{A_1(\CC)}(\varphi_N) = {}_{A_1(\CC)} \langle \mathcal{G} \rangle:= \{ \mathcal {A} \mathcal{G} \, : \,  \mathcal{A} \in A_1 (\CC)\}.
\end{equation*}

In fact, following Proposition \ref{prop:PSO_General_Targets} part (d), it is generally true that, for a real-valued random variable $X$, $\PSO(X)$ is a principal right ideal if and only if the annihilator $\Ann_{A_1 (\CC)} (\varphi_X)$ is a left principal ideal in which the generators are linked via the Fourier map $\Psi$.

The annihilator ideal $\Ann_{A_1(\CC)}(\varphi_N) $ plays a major role in the characterising Stein operators. Obviously, every Stein operator $L \in \PSO(N)$ corresponds to a unique element $\mathcal{L} \in \Ann_{A_1(\CC)}(\varphi_N) $ and vice versa via the Fourier map $\Psi$. Let us continue with the following simple observation.

\begin{prop}\label{prop:General_Characterization}
Let $ 0 \neq L \in \mathrm{PSO}(N)$ and $\mathcal{L} = \Psi (L)$ be the image of the operator $L$ under the Fourier map $\Psi$ (introduced in Lemma \ref{lem:Anti-Isomorphism}).  Then $L$ is a characterising Stein operator for the standard Gaussian distribution in the sense of Definition \ref{def:Charactrization} if and only if

\begin{equation*}
	\mathrm{Ker} (\mathcal{L}) \cap \widehat{\mathcal{P}}(\R) = \{  e^{-t^2/2}\},
\end{equation*}
where $\widehat{\mathcal{P}}(\R)$ is the class of characteristic functions of probability distributions on the real line.
\end{prop}

\begin{proof}
	It is immediate that $\varphi_N (t) = e^{-t^2/2} \in \Ker (\mathcal{L})$ because $\mathcal{L} = \mathcal{A}(t+\partial)$ for some $\mathcal{A} \in A_1 (\CC)$ and $\mathcal{G} \varphi_N(t)= (t+\partial)e^{-t^2/2}=0$. In addition, $e^{-t^2/2} \neq \varphi \in \Ker (\mathcal{L})$ if and only if $\E[L f (X)]=\E[L f (N)]=0$ for all $f \in \mathcal{F}$, where $X$ is a random variable whose characteristic function is $\varphi$. Hence, the claim follows. 
\end{proof}

We continue with some illustrating examples. These examples show how one can construct polynomial Stein operators for the standard Gaussian of order $T\geq2$ that are not characterising.

\begin{ex}\label{ex:Moment_Assumption_Crucial}\label{extocor}{\rm 
(a)  Let $G=x-\partial  \in \mathrm{PSO}(N)$. Then, obviously, $G$ is a characterising Stein operator for the standard Gaussian distribution. This goes all the way back to Stein's seminal paper \cite{stein}. In order to apply Proposition \ref{prop:General_Characterization}, note that $\Psi (G) = \mathcal{G} = t + \partial$ and, clearly $\Ker (\mathcal{G}) \cap \widehat{\mathcal{P}}(\R) = \{  e^{-t^2/2}\}$.

\vspace{2mm}

\noindent{(b)} Let $ x_0 \in \R\setminus\{0\}$, and consider the operator $\mathcal{L} = \mathcal{A} (t+\partial) \in \mathrm{Ann}_{A_1(\CC)}(\varphi_N)$, where $\mathcal{A} = (t - \Ci x_0) +\partial \in A_1 (\CC)$. One can easily check that both characteristic functions 
\begin{equation*}
e^{-t^2/2}, \, e^{-t^2/2} e^{\Ci t x_0} \in \Ker (\mathcal{L}) \cap \widehat{\mathcal{P}}(\R).  
\end{equation*}
 These characteristic functions correspond to the $N(0,1)$ and $N(x_0,1)$ distributions, respectively. Hence, Proposition \ref{prop:General_Characterization} entails that the operator $L = \Psi^{-1}(\mathcal{L})$ is not a characterising Stein operator for the standard Gaussian distribution without assuming a further condition on the first derivative (evaluated at $t=0$) of characteristic functions $\varphi \in \widehat{\mathcal{P}}(\R) $ (note that $-\Ci \, \frac{\partial}{\partial t} (  e^{-t^2/2} e^{\Ci t x_0}  )\big \vert_{t=0}= x_0 \neq 0$). For example, with the additional assumption $\E[X]=0$ (that is $\varphi'(0)=0$), the operator $L$ is a characterising polynomial Stein operator for the $N(0,1)$ distribution.

\vspace{2mm}

\noindent{(c)} This time, let	$	{\mathcal A}=( t \partial + 2 t^2 -1  ) $  and   $ {\mathcal L}={\mathcal A}(t+\partial)
		=(t \partial^2 +3 t^2 \partial - \partial + 2 t^3 )$.
	Then, one can easily check that  the characteristic functions
\begin{equation*}
e^{-t^2/2}, \,   e^{-t^2} \in \Ker (\mathcal{L}) \cap \widehat{\mathcal{P}}(\R) 
\end{equation*}
corresponding to the Gaussian distributions $N(0,1)$ and $N(0,2)$, respectively. Hence, operator $L=\Psi^{-1}(\mathcal{L})$ is not a characterising Stein operator for the standard Gaussian distribution without assuming a further condition on the second derivative (evaluated at $t=0$) of characteristic functions $\varphi \in \widehat{\mathcal{P}}(\R) $ (note that $-\frac{\partial^2}{\partial t^2} ( e^{-t^2})\big \vert_{t=0}=2$). For example, with the additional assumption $\E[X^2]=1$ (that is $\varphi''(0)=-1$), the operator $L$ is a characterising polynomial Stein operator for the $N(0,1)$ distribution.

		\vspace{2mm}
		
		\noindent{(d)} More generally, there exist $T$-th order ($T\geq2$) differential operators in the class $\mathrm{PSO}(N)$ for which $T-1$ moment conditions must be specified for the operator $S$ to be characterising. Consider a mixture of Gaussian distributions $Y$ with density $p_Y(x)=\sum_{j=1}^T w_jp_j(x)$, where $w_1,\ldots,w_T>0$, $\sum_{j=1}^Tw_j=1$ and, for $j=1,\ldots,T$, the function $p_j$ is the density of a $N(0,\sigma_j^2)$ random variable, with $\sigma_1^2,\ldots,\sigma_T^2>0$ distinct and also $\sigma_1^2=1.$ Since $p_1,\ldots,p_T$ are linearly independent holonomic functions that satisfy first order homogeneous differential equations with polynomial coefficients, it follows that $p_Y(x)=\sum_{j=1}^T w_jp_j(x)$ is a holonomic function and satisfies a $T$-th order homogeneous differential equation of the form $Lp=0$, where $Lp(x)=\sum_{j=0}^m\sum_{k=0}^T a_{jk}x^jp^{(k)}(x)$, with real-valued coefficients $a_{jk}$ that do not depend on $w_1,\ldots,w_T$ (they do not depend on $w_1,\ldots,w_T$ because any linear combination of $p_1,\ldots,p_T$ will satisfy $Lp=0$). By a standard integration by parts argument (see, for example, the proof of Lemma 4.1 of \cite{gms19}) it follows that a Stein operator for $Y$ is given by $Sf(x)=\sum_{j=0}^m\sum_{k=0}^T (-1)^ka_{jk}\partial^k(x^jf(x))$, where $f\in P(\R)$. Since the operator $S$ has no dependence on $w_1,\ldots,w_T$, for $S$ to be a characterising Stein operator for $X\sim N(0,1)$, additional conditions must be specified to ensure that $w_1=1$ and $w_2=\cdots=w_T=0$. For a characterisation based on moment conditions, $T-1$ moments of $X$ would need to be specified; we do not need to specify $T$ moments because $w_1=1-\sum_{j=2}^Tw_j$.
		
		
}
\end{ex}

The next proposition studies the structure of the annihilator ideal $\Ann_{A_1(\CC)}(\varphi_N) $. It turns out that the null space of the annihilator ideal $\Ann_{A_1(\CC)}(\varphi_N) $ is extremely large and in fact it includes any holonomic characteristic function.  

\begin{prop}\label{prop:Ideal-Structure-In-Fourier-Side}
Let $$\mathrm{Ker} \big(  \mathrm{Ann}_{A_1(\CC)}(\varphi_N)   \big) : = \{ \mathrm{Ker} (\mathcal{L}) \, :\, \mathcal{L} \in \mathrm{Ann}_{A_1(\CC)}(\varphi_N) \} = \bigcup_{\mathcal{A} \in A_1 (\CC)} \mathrm{Ker} \left( \mathcal{A} (t+\partial) \right).$$ Then
\begin{equation*}
\mathrm{Ker} \big(  \mathrm{Ann}_{A_1(\CC)}(\varphi_N)   \big) \cap \widehat{\mathcal{P}}(\R) =  \widehat{\mathcal{P}}(\R) ^{\mathrm{Hol}},
\end{equation*}
where $\widehat{\mathcal{P}}(\R) ^{\mathrm{Hol}}$ stands for the class of all holonomic characteristic functions (\textbf{possibly discrete distributions, e.g. the Bernoulli distribution}).
\end{prop}

\begin{proof}
It is clear that $ \Ker (  \Ann_{A_1(\CC)}(\varphi_N)   ) \cap \widehat{\mathcal{P}}(\R) \subseteq \widehat{\mathcal{P}}(\R) ^{\mathrm{Hol}}$. For the other inclusion, let $\varphi \in \widehat{\mathcal{P}}(\R) ^{\mathrm{Hol}}$. Then, there is $\mathcal{L} \in A_1 (\CC)$ such that $\mathcal{L} \varphi=0$. On the other hand, by the left Ore property, one can write $ \mathcal{S} \mathcal{L} = \mathcal{A}(t+\partial)$ for some elements $\mathcal{S}, \mathcal{A} \in A_1 (\CC)$. Hence, $\mathcal{A}(t+\partial) \in \Ann_{A_1(\CC)} (\varphi_N)$ and $ \mathcal{A}(t+\partial) \varphi =0$.  Therefore, $\varphi \in \Ker ( \mathcal{A}(t+\partial) )$ and we are done.
\end{proof}

\subsection{Characterising Stein operators for the Gaussian distribution
}\label{sec4.3}

In the following proposition, we show that first order polynomial Stein equations for the Gaussian distribution are characterising in the sense of Definition \ref{def:Charactrization}. An immediate consequence is that Stein operator $S_m : = H_{m-1} \partial - H_{m} \in \mathrm{PSO} (N)$, $m\geq1$, of \cite{gr05} is characterising.



\begin{prop}\label{cor:Hermite-Coeff-Is-Characterizing}
Let $q \in \R[x]$ be a non-zero polynomial with real coefficients and $\delta$ stands for the adjoint operator. Then, the operator 
\[L = q(x) \partial - \delta q(x) \in \mathrm{PSO} (N)\]
is a characterising polynomial Stein operator for the standard Gaussian distribution in the sense of Definition \ref{def:Charactrization}.

In particular,
the operator $S_m : = H_{m-1} \partial - H_{m} \in \mathrm{PSO} (N)$, $m\geq1$, is a characterising polynomial Stein operator for the standard Gaussian distribution.
\end{prop}

\begin{proof}
   We first note the following simple connection between the Stein operator $q(x) \partial - \delta q(x)$ and the classical Stein operator $\partial - x$. Consider the Stein operator $f'(x) - xf(x)$. Let $f(x)=q(x)g(x)$. Then $f'(x) - xf(x) = q(x) g'(x) -\delta q(x) g(x)$. Now, let $X$ be a real-valued random variable with density $p_X(x)$ and suppose that $\E[q(X) g'(X) - \delta q(X) g(X)]=0$ for all $g\in \mathcal{G}$ for some suitable class of functions $\mathcal{G}$. Then, by an integration by parts argument, we have that $\int_\R (p_X'(x)+xp_X(x)) q(x)g(x) \,dx=0$ for all $g\in \mathcal{G}$. We thus deduce that $p_X(x)$ satisfies the ODE $p_X'(x)+xp_X(x)=0$. On solving this ODE subject to the condition that $p_X(x)$ is the probability density function of a real-valued random variable (so that $\int_\R p_X(x)\,dx=1)$ we deduce that $p_X(x)=(2\pi)^{-1/2}e^{-x^2/2}$, and therefore $X\sim N(0,1)$. 

That the operator $S_m : = H_{m-1} \partial - H_{m} \in \mathrm{PSO} (N)$ is a characterising polynomial Stein operator for the standard Gaussian distribution now follows on taking $q(x)=H_{m-1}(x)$ and using that $\delta H_{m-1}(x)=H_m(x)$ (see Definition \ref{def:Hermite_Polynomials}).
\end{proof}

We now consider the characterising property of higher order Stein operators. In the light of part (d) of Example \ref{ex:Moment_Assumption_Crucial},
we know that there exist many higher order Stein operators that are not characterising. Our aim here is to attempt to give some examples of general classes of higher order polynomial Stein operators for the Gaussian distribution that are characterising.  The most natural operator to start with is perhaps the higher order counterpart of the Stein operator $S_m$ given in Proposition \ref{cor:Hermite-Coeff-Is-Characterizing}, namely the Stein operator  $L_m=\partial^m - H_m (x)$, $m\geq1$. As noted in the introduction, the fact that $L_m\in\mathrm{PSO}(N)$, $m\geq1$, was known to \cite{gr05}. The proof of the following proposition is technical and is given in Appendix \ref{appa}.

\begin{prop}\label{prop4.5}
	Let $ m\ge 1$ be fixed.  Then, the operator $L_m = \partial^m - H_m  \in \mathrm{PSO}(N) $
is a characterising polynomial Stein operator for the standard Gaussian distribution among those random variables that are symmetric (that is $X=_d-X$) and have an infinitely divisible distribution. When $m=1,2,3$, the characterisation holds without the need for the symmetric and infinitely divisible distribution assumptions.
\end{prop}

\begin{rem}\rm{In Proposition \ref{prop4.5}, our characterisation of the standard Gaussian distribution has additional symmetry and infinite divisibility conditions when $m\geq4$. 
We conjecture that these additional conditions are not necessary to ensure characterisation, but are rather artefacts resulting from our proof. We leave it as an interesting open problem to verify or disprove this conjecture. Similar comments apply for Corollary \ref{barops} and Proposition \ref{prop4.6} below.}
\end{rem}

We can generalise Proposition \ref{prop4.5} to deduce a characterisation result for the Stein operators $S(k,t)=H_{k}\partial^t-H_{k+t}\in\PSO(N)$, $k,t\geq1$. Recall that the fact that $S(k,t)\in\mathrm{PSO}(N)$, $k,t\geq1$, was known to \cite{b86}.

\begin{cor}\label{barops}
Let $k, t\geq1$. Then, the operator $H_k \partial^t - H_{k+t} \in \mathrm{PSO}(N)$ is a characterising Stein operator for the standard Gaussian distributions  among those random variables that are symmetric and have an infinitely divisible distribution. When $t=1,2,3$, the characterisation holds without the need for the symmetric and infinitely divisible distribution assumptions. 
\end{cor}

\begin{proof}
We deduce the result from Proposition \ref{prop4.5} by applying the technique used in the proof of Proposition \ref{cor:Hermite-Coeff-Is-Characterizing}. Consider the Stein operator $\partial^t f(x) - H_t(x)f(x)$. Now, let $X$ be a real-valued random variable with density $p_X(x)$ and suppose that $\E[ f^{(t)}(X) - H_{t}(X) f(X)]=0$ for all $f\in \mathcal{F}$ for some suitable class of functions $\mathcal{F}$. Then, by an integration by parts argument, we have that $\int_\R  f(x) Tp_X(x) \,dx=0$ for all $f\in \mathcal{F}$, for some linear differential operator $T$ (the particular form that the operator $T$ takes is not important to our argument). We thus deduce that $p_X(x)$ satisfies the ODE $Tp_X(x)=0$. By proposition \ref{prop4.5}, we can infer that the only solution to this ODE that satisfies the properties of a probability density function is $p_X(x)=(2\pi)^{-1/2}e^{-x^2/2}$ if $t=1,2,3$, and again $p_X(x)=(2\pi)^{-1/2}e^{-x^2/2}$ if we impose the additional assumptions that $p_X(x)$ must be the density of a symmetric and infinitely divisible random variable when $t\geq4$.

Now, let $f(x)=H_k(x)g(x)$. Then we get that $\partial^t f(x) - H_t(x)f(x) = H_k(x)\partial^t g(x) - H_{k+t}(x)g(x)$.
Let $Y$ be a real-valued random variable with density $p_Y(x)$ and suppose that $\E[H_{k}(Y) g^{(t)}(Y) - H_{k+t}(Y) g(Y)]=0$ for all $g\in \mathcal{G}$ for some suitable class of functions $\mathcal{G}$. Then, by the same integration by parts argument as before, we have that $\int_\R  H_{k}(x)g(x)Tp_Y(x) \,dx=0$ for all $g\in \mathcal{G}$, where $T$ is the same operator as introduced earlier in this proof. We thus deduce that $p_Y(x)$ satisfies the ODE $Tp_Y(x)=0$, the same ODE as satisfied by $p_X(x)$. Thus, under the relevant assumptions for the cases $t=1,2,3$ and $t\geq4$, the only solution that satisfies the properties of a probability density function is $p_Y(x)=(2\pi)^{-1/2}e^{-x^2/2}$, and therefore $Y\sim N(0,1)$. 
\end{proof}
	

Another interesting class of higher order Stein operators for the standard Gaussian distribution  can be derived via the so-called Rodriguez formula for Hermite polynomials. In fact,   the Rodriguez formula  implies that the characteristic function $\varphi_N (t) = e^{-t^2/2}$ satisfies the ODE
\begin{equation}\label{eq:Rodriguez-formula-Hermite}
\partial^m \varphi_N ( t)  = (-1)^m H_m (t) \varphi_N(t), \quad \forall \, m \ge 1.
 \end{equation}
Let $\mathcal{R}_m  = \partial^m - (-1)^m H_m \in \mathrm{Ann}_{A_1 (\CC)} (\varphi_N)$ and  hence $R_m = \Psi^{-1}(\mathcal{R}_m)=(-1)^mH_m(-\mathrm{i}\partial)-\mathrm{i}^mx^m \in \PSO(N)$ (possibly after a multiplication by a suitable power of $\mathrm{i}$ to get real coefficients). The polynomial Stein operators $R_m$ are particularly interesting
whenever the sole $m$-th  moment matching $\E[X^m] = \E[N^m]$ is applicable. The proof of the following proposition is given in Appendix \ref{appa}. 



\begin{prop}\label{prop4.6}Let $ m\ge 1$ be fixed. 
Then, the operator $R_m=(-1)^mH_m(-\mathrm{i}\partial)-\mathrm{i}^mx^m \in \mathrm{PSO}(N) $ is a characterising polynomial Stein operator for the standard Gaussian distribution among those random variables that are symmetric and have an infinitely divisible distribution. When $m=1,2,3$, the characterisation holds without the need for the symmetric and infinitely divisible distribution assumptions.
\end{prop}

\section{A general result on the characterising feature of polynomial Stein operators}\label{secnew}

In Section \ref{sec4}, we saw examples of non-trivial elements of $\PSO(N)$ that are not characterising in the sense of Definition \ref{def:Charactrization}. This non-characterising property in fact applies for any distribution with a holonomic density or holonomic characteristic function, and is a consequence of the following general theorem.


\begin{thm}\label{thm:PSO_Intersection_NonTrivial}
	Fix $n \ge 1$. Let $X_1, \ldots, X_n$	be continuous random variables. Suppose that the characteristic functions of $X_1,\ldots,X_n$ are holonomic, or $X_1, \ldots, X_n$	have holonomic densities. Then
	\begin{equation}\label{eq:POS-Intersection-General}
		\mathrm{PSO}(X_1) \cap \cdots \cap \mathrm{PSO}(X_n) \neq \{0\}.
	\end{equation}
\end{thm}

\begin{proof}
	First note that the assumption that the characteristic functions of $X_1,\ldots,X_n$ are holonomic implies that $\PSO(X_k) \neq \{0\}$ for each $k=1,\ldots,n$, by Proposition \ref{prop:PSO_General_Targets}, item (a). Similarly, the assumption that densities of $X_1,\ldots,X_n$ are holonomic  implies that $\PSO(X_k) \neq \{0\}$ for each $k=1,\ldots,n$, by Proposition \ref{prop:PSO_General_Targets}, item (b). On the other hand, by Stafford's Theorem \ref{thm:Stafford} each right ideal $\PSO(X_k) = \langle G^k_1, G^k_2 \rangle$  for two generators $G^k_1$ and $G^k_2$. Since Weyl algebras are right Noetherian integral domains they are right Ore domains \cite[Theorem 1.15]{mr-Book87}. Therefore, letting $R = A_1(\R)$, we have
	\begin{equation*}
		\bigcap_{k=1}^{n} G^k_1 R \cap G^k_2 R  \neq \{0\}.
	\end{equation*} 
	On the other hand, $ G^k_1 R \cap G^k_2 R  \subseteq \PSO(X_k)$ for each $k=1,\ldots,n$, and whence the claim follows.
\end{proof}

\begin{cor}\label{cor:PSO_Intersection_NonTrivial}
	Let $N \sim N(0,1)$ and $X$ be a continuous random variable whose characteristic function is holonomic, or whose density is holonomic. Then
	\begin{equation*}
		\mathrm{PSO}(N) \cap \mathrm{PSO}(X) \neq \{0\}.
	\end{equation*} 
\end{cor}

\begin{rem}\label{rem4.2}{  \rm (a) In the statement of Theorem \ref{thm:PSO_Intersection_NonTrivial}, the supports of the random variables $X_1,\ldots,X_n$ do not need to be the same. In fact, one can even consider random variables having disjoint supports. Let $X$ and $Y$ be two continuous random variables whose characteristic functions are holonomic or whose densities are holonomic with disjoint supports $\mathbb{D}_X$ and $\mathbb{D}_Y$, respectively. Let $\mathbb{U}_X$ and $\mathbb{U}_Y$ be two disjoint open domains so that $\mathbb{U}_X \supseteq   \mathbb{D}_X$, $ \mathbb{U}_Y \supseteq   \mathbb{D}_Y$. It is evident that $\E[S_X  f(Y)]=0$ for all $f\in \mathcal{S}(\mathbb{U}_X)$, where $S_X \in \PSO(X)$. In addition, clearly $ \mathcal{S}(\mathbb{U}_X)  \subseteq \mathcal{S}(\R)$; however, the requirement $\E[S_X f(Y)]=0$ for all $f \in \mathcal{S}(\R)$ is a non-trivial task. According to Theorem \ref{thm:PSO_Intersection_NonTrivial} there is at least one non-trivial polynomial differential operator $ S \neq 0$ such that
			\begin{equation*}
				\E[S f(X_k)]=0, \quad \forall \, f \in \mathcal{S}(\R) \supseteq \mathcal{S} (\mathbb{U}), \quad k=1,\ldots,n,
			\end{equation*} 
			where each $\mathbb{D}_{X_k}$ is the support of random variable $X_k$, and the open domain $\mathbb{U} \supseteq \mathbb{D}_{X_1} \cup \cdots \cup \mathbb{D}_{X_n}$.  An example of an operator $0 \neq L \in \PSO (N) \cap \PSO(X)$ is given in \cite[Section 3.1]{a-g-g-Stein-charactrization} when $X$ is distributed as a semi-circular distribution with probability density function $p_X(x)= (2/\pi) \sqrt{1-x^2} \textbf{1}_{[-1,1]}$ supported on the interval $[-1,1]$.
			
   \vspace{2mm}
   
   \noindent{(b)}			  Theorem \ref{thm:PSO_Intersection_NonTrivial} is also related to characterising Stein operators; see Definition \ref{def:Charactrization}. Firstly, in order that a polynomial Stein operator would be useful from a probabilistic point of view, the class $\mathcal{F}$ of test functions must contain functions being non-zero on the support of the underlying target random variable. Secondly, for random variables $X$ and $Y$ as in part (a) of this remark, Theorem \ref{thm:PSO_Intersection_NonTrivial} yields the existence of a polynomial Stein operator $S\not=0$ so that $\E[Sf(X)]=\E[Sf(Y)]=0$ for all $f \in \mathcal{S}(\R)$. Therefore, the Stein operator $S$ cannot be a characterising operator for neither $X$  nor $Y$ without any extra assumption; see also Example \ref{ex:Moment_Assumption_Crucial}.
	} 
\end{rem}

\appendix

\section{Further proofs}\label{appa}

\noindent{{\bf{Proof of Proposition \ref{prop4.5}.}}} The case $m=1$ corresponds to the classical standard Gaussian Stein operator $\partial-x$, so we focus on the case $m\geq2$. Recall that $L_m=\partial^m-H_m(x)$. Then applying the mapping $\mathcal{L}_m=\Psi(L_m)$
we have that $\mathcal{L}_m=\mathrm{i}^mt^m-H_m(-\mathrm{i}\partial)$, and $\mathcal{L}_m\in\mathrm{Ann}_{A_1(\mathbb{C})}(\varphi_N)$. We shall employ the technique for proving sufficiency of Stein characterisations from our recent paper \cite{a-g-g-Stein-charactrization} to prove that $\mathrm{Ker}(\mathcal{L}_m)\cap \widehat{\mathcal{P}} (\R) =\{e^{-t^2/2}\}$ for $m=2,3$ under no additional assumptions on the random variable $X$, and for $m\geq4$ under the assumption that the random variable $X$ is symmetric and has an infinitely divisible distribution.

Consider the ODE 
\begin{equation}\label{ode}\mathcal{L}_m\varphi(t)=0.
\end{equation}
 This ODE has an essential singularity at $t=\infty$, and so we seek asymptotic approximations for the solutions using the ansatz $\varphi(t)=e^{S(1/t)}$. Using this ansatz and the method of dominate balance (see \cite[Chapter 3]{bender} for a detailed account of this method) we deduce that $S'(1/t))^m\sim t^{3m}$, as $t\rightarrow\infty$.
We thus deduce that there are $m$ linearly independent solutions to the ODE (\ref{ode}) with the following asymptotic behaviour as $t\rightarrow\infty$,
\begin{equation*}\varphi_{j,m}(t)=\exp\bigg(-\frac{\omega_m^{j-1}t^2}{2}(1+o(1))\bigg), \quad j=1,\ldots,m,
\end{equation*}
where $\omega_m$ is the $m$-th root of unity (recall that $1,\omega_m,\ldots, \omega_m^{m-1}$ are the $m$ distinct roots of the polynomial $x^m-1$). The general solution of the ODE (\ref{ode}) is given by 
\begin{equation}\label{gensol}\varphi(t)=\sum_{j=1}^{m}C_{j,m}\varphi_{j,m}(t)
\end{equation}
 for some arbitrary constants $C_{1,m},\ldots,C_{m,m}$. Note that $\varphi_{1,m}(t)=\exp(-(t^2/2)(1+o(1)))$, as $t\rightarrow\infty$; this asymptotic behaviour is consistent with that of the characteristic function $e^{-t^2/2}$ of the standard Gaussian distribution. If we can argue that $\varphi_m(t)$ can only be a characteristic function if $C_{2,m}=\cdots=C_{m,m}=0$, then it follows that $\mathrm{Ker}(\mathcal{L}_m)\cap \widehat{\mathcal{P}} (\R) =\{e^{-t^2/2}\}$, meaning that $L_m$ is a characterising Stein operator.

We first consider the case $m=2$. In this case, $\varphi_{2,2}(t)=\exp((t^2/2)(1+o(1)))$, as $t\rightarrow\infty$, and so this solution is unbounded as $t\rightarrow\infty$. Since characteristic functions are bounded, in order for $\varphi_2(t)$ to be a characteristic function we must take $C_{2,2}=0$, and so $L_2$ is a characterising Stein operator. Similarly, the solutions $\varphi_{2,3}(t)$ and $\varphi_{3,3}(t)$ blow up as $t\rightarrow\infty$, and so we deduce by the same argument that $L_3$ is a characterising Stein operator.

Suppose now that $m\geq4$. It is no longer the case that all solutions except for $\varphi_{1,m}(t)$ are unbounded as $t\rightarrow\infty$, and so the argument used to treat the cases $m=2,3$ is no longer enough to deduce that the operator $L_m$ is characterising. To deal with this difficulty, we introduce the additional assumptions that the distribution of $X$ is symmetric and infinitely divisible. We shall exploit the properties that the characteristic function $\varphi(t)$ of an infinitely divisible distribution is non-zero for all $t$, and that the characteristic function of a symmetric random variable is real-valued.

Suppose $m\geq4$ is even; the odd $m$ case is similar, and so is omitted. 
Let $R_{j,m}$ and $I_{j,m}$ denote the real and imaginary parts of $\omega^{j-1}$ for $j=1,\ldots,m$. Then applying Euler's formula $e^{\mathrm{i}x}=\cos(x)+\mathrm{i}\sin(x)$ to the solution (\ref{gensol}) gives the following alternative representation of the general solution:
  \[\varphi_m(t)=\sum_{j=1}^m a_{j,m} g_{j,m}(t),\]
where $a_{1,m},\ldots,a_{m,m}$ are arbitrary constants, and, as $t\rightarrow\infty$,
\begin{align*}g_{1,m}(t)&=\exp\bigg(-\frac{t^2}{2}(1+o(1))\bigg), \quad g_{m/2,m}(t)=\exp\bigg(\frac{t^2}{2}(1+o(1))\bigg), \\
g_{j,m}(t)&=\sin\bigg(\frac{I_{j,m}t^2}{2}(1+o(1))\bigg)\exp\bigg(-\frac{R_{j,m}t^2}{2}(1+o(1))\bigg), \quad j=2,\ldots,m/2-1,\\
g_{j,m}(t)&=\cos\bigg(\frac{I_{j,m}t^2}{2}(1+o(1))\bigg)\exp\bigg(-\frac{R_{j,m}t^2}{2}(1+o(1))\bigg), \quad j=m/2+1,\ldots,m.
\end{align*}
Now, $g_{m/2}(t)\rightarrow\infty$ as $t\rightarrow\infty$, so we must take $a_{m/2,m}=0$. We note that $I_{j,m}\not=0$ and $R_{j,m}<1$ for $j\in\{2,\ldots,m/2-1\}\cup\{m/2+1,\ldots,m\}$. Moreover, the symmetric distribution assumption ensures that the solution $\varphi_m(t)$ is real-valued. Due to the infinitely divisible distribution assumption, $\varphi_m(t)$ cannot equal to zero for any $t\in\R$, and we are now forced to take $a_{2,m}=\cdots=a_{m,m}=0$, which implies that the Stein operator $L_m$ is characterising.
\hfill $\Box$

\vspace{3mm}

\noindent{{\bf{Proof of Proposition \ref{prop4.6}.}}} Following the same steps as in the proof of Proposition \ref{prop4.5} yields that the general solution to the ODE $\mathcal{R}_m\varphi_m(t)=0$ is given by 
$\varphi_m(t)=\sum_{j=1}^mB_{j,m}\tilde{\varphi}_{j,m}(t)$
for some arbitrary constants $B_{1,m},\ldots,B_{m,m}$, and, as $t\rightarrow\infty$, 
\begin{equation*}\tilde{\varphi}_{j,m}(t)=\exp\bigg(-\frac{\omega_m^{j-1}t^2}{2}(1+o(1))\bigg), \quad j=1,\ldots,m.
\end{equation*}
To leading order, the linearly independent solutions $\tilde{\varphi}_{j,m}(t)$, $j=1,\ldots,m$, have precisely the same asymptotic behaviour as the linearly independent solutions $\varphi_{j,m}(t)$, $j=1,\ldots,m$, from the proof of Proposition \ref{prop4.5}. Reasoning exactly as we did in the proof of Proposition \ref{prop4.5} now completes the proof.
\hfill $\Box$

\subsection*{Acknowledgements}
Parts (a) and (b) of Proposition \ref{prop:PSO_General_Targets}, and Proposition \ref{holprop} were inspired by an excellent suggestion of one of the reviewers of our paper \cite{a-g-g-algebraic-stein-operators}, who commented on the connection between Stein's method and holonomic function theory. RG is funded in part by EPSRC grant EP/Y008650/1 and and EPSRC grant UKRI068. 




\begin{thebibliography}{AMMP16}

\bibitem[Ana23]{steinstat}
\textsc{Anastasiou, A. et al.} (2023)
\newblock  Stein's Method Meets Computational Statistics: A Review of Some Recent Developments. \newblock \emph{Stat. Sci.} $\mathbf{38}$ (2023),  120-139.

\bibitem[AGG23a]{a-g-g-algebraic-stein-operators}
\textsc{Azmoodeh, A., Gasbarra, D., Gaunt. R. E.} (2023)
\newblock On algebraic Stein operators for Gaussian polynomials.
\newblock \emph{Bernoulli} $\mathbf{29}$, 350--376.

\bibitem[AGG23b]{a-g-g-Stein-charactrization}
\textsc{Azmoodeh, A., Gasbarra, D., Gaunt. R. E.} (2023)
\newblock An asymptotic approach to proving sufficiency of Stein characterisations.
\newblock \emph{ALEA Lat. Am. J. Probab. Stat.} $\mathbf{20}$, 127--152.

\bibitem[Bar86]{b86}
\textsc{Barbour, A. D.} (1986)
\newblock Asymptotic Expansions Based on Smooth Functions
in the Central Limit Theorem.
\newblock \emph{Probab. Th. Rel. Fields} $\mathbf{72}$, 289--303.

\bibitem[Bar88]{b88}
\textsc{Barbour, A. D.} (1988)
\newblock Stein's Method and Poisson Process Convergence. 	
\newblock \emph{J. Appl. Probab.} 
$\mathbf{25}$A, 175--184. 

\bibitem[Bar90]{barbour2} 
\textsc{Barbour, A. D.} (1990)
\newblock Stein's method for diffusion approximations.
\newblock  \emph{Probab. Th. Rel. Fields} $\mathbf{84}$, 297--322.

\bibitem[BHJ92]{bhj92} 
\textsc{Barbour, A. D., Holst, L., Janson, S.} (1992)
\newblock \emph{Poisson Approximation}. Oxford University Press, Oxford.

\bibitem[BO99]{bender}
\textsc{Bender, C. M., Orszag, S. A.} (1999).
\newblock \emph{Advanced Mathematical Methods for Scientists and Engineers I: Asymptotic Methods and Perturbation Theory.}
\newblock Springer.

\bibitem[Bj\"o79]{Bjork-Book}
\textsc{Bj\"ork, J.-E.} (1979)
\newblock \emph{Rings of Differential Operators}. North-Holland Mathematical Library, 21. North-Holland Publishing Co., Amsterdam-New York.


\bibitem[BCLS]{BoundOnOrder}
\textsc{Bostan, A., Chyzak, F., Li, Z., 
Salvy, B.} (2012)
\newblock Fast computation of common left multiples of
linear ordinary differential operators. 
\newblock In \emph{Proceedings of
ISSAC'12},  99--106.


\bibitem[CGS11]{chen} 
\textsc{Chen, L. H. Y., Goldstein, L., Shao, Q.-M.} (2011)
\newblock  \emph{Normal Approximation by Stein's Method.}
\newblock Springer.





\bibitem[CF16]{Gumber_Stein}
\textsc{Cipriani, A., Feidt, A.}  (2016)
\newblock Rates of convergence for extremes of geometric random variables and marked point processes.  
\newblock	\emph{Extremes} $\mathbf{19}$, 105--138.
	
	
	
	
	
	
\bibitem[Cou95]{Coutinho-Book}
\textsc{Coutinho, S. C.} (1995)
\newblock \emph{A Primer of Algebraic D-Modules}.
\newblock London Mathematical Society Student Texts, \textbf{33}, Cambridge University Press.



\bibitem[CF08]{cf-Book}
\textsc{Cozzens, J., Faith, C.} (2008)
\newblock \emph{Simple Noetherian Rings}. 
\newblock Cambridge Tracts in Mathematics, 69. Cambridge University Press, Cambridge.
	
	
	
	
	
	
	





\bibitem[EFGPS]{ebner} 
\textsc{Ebner, B., Fischer, A., Gaunt, R. E., Picker, B. and Swan, Y.} (2023)
\newblock Stein's Method of Moments. arXiv:2305.19031, 2023.

\bibitem[FGRS22]{fathi} 
\textsc{Fathi, M., Goldstein, L., Reinert, G. and Saumard, A.} (2022)
\newblock Relaxing the Gaussian assumption in shrinkage and SURE in high dimension. 
\newblock \emph{Ann. Stat.} $\mathbf{50}$ (2022), 2737--2766.

\bibitem[FGS05]{Non-Holonomic-Character}
\textsc{Flajolet, P.,  Gerhold, S., Salvy, B.} (2005)
\newblock On the non-holonomic character of logarithms, powers,
and the $n$th prime function.
\newblock \emph{Electron. J. Comb.} $\mathbf{11}$(2) A2, 1--16.

\bibitem[FS09]{Flajolet-Book}
\textsc{Flajolet, P., Sedgewick, R.} (2009)
\newblock \emph{Analytic Combinatorics}.
\newblock  Cambridge University Press, Cambridge.

\bibitem[FR22]{fr22}
\textsc{Fulman, J., R\"{o}llin, A.} (2022) 
\newblock Stein's method, heat kernel, and linear functions on the orthogonal groups. 
\newblock \emph{J. Algebra} $\mathbf{607}$, 272--285.

\bibitem[GMS19]{gms19}
\textsc{Gaunt, R. E., Mijoule, G., Swan, Y.} (2019)
\newblock  An algebra of Stein operators. 
\newblock \emph{J. Math. Anal. Appl.} $\mathbf{469}$, 260--279.

\bibitem[Gil51]{g51}
\textsc{Gil-Pelaez, J.} (1951)
\newblock Note on the Inversion Theorem.
\newblock \emph{Biometrika} $\mathbf{38}$, 481--482.


\bibitem[GW04]{gw-Book}
\textsc{Goodearl, K. R., Warfield, R. B.} (2004) 
\newblock \emph{An Introduction to Noncommutative Noetherian Rings}. 
\newblock Second edition. London Mathematical Society Student Texts, 61. Cambridge University Press, Cambridge.
	
\bibitem[GR05]{gr05}
\textsc{Goldstein, L.,  Reinert, G.} (2005)
\newblock	Distributional transformations, orthogonal polynomials, and Stein characterizations.  
\newblock \emph{J. Theoret. Probab.} $\mathbf{18}$, 237--260. 

\bibitem[G\"{o}t91]{gotze} 
\textsc{G\"{o}tze, F.} (1991)
\newblock  On the rate of convergence in the multivariate CLT.  \newblock \emph{Ann. Probab.} $\mathbf{19}$,  724--739.

\bibitem[Har09]{h09} 
\textsc{Harper, A. J.} (2009).
\newblock Two new proofs of the Erd\"os–Kac theorem, with bound on the rate of convergence, by Stein's method for distributional approximations. 
\newblock \emph{Math. Proc. Cambridge} $\mathbf{147}$, 95--114.

\bibitem[Inc56]{Ince-Book}
\textsc{Ince, E. L.} (1956).
\newblock \emph{Ordinary Differential Equations}. \newblock New York: Dover.

\bibitem[Kau13]{Holonomic-Toolkit}
\textsc{Kauers, M.} (2013)
\newblock The Holonomic Koolkit.  
\newblock \emph{Computer Algebra in Quantum Field Theory} 119--144, 
Texts Monogr. Symbol. Comput., Springer, Vienna. 

\bibitem[KBY94]{k-b-Ustatistics}
\textsc{Koroljuk, V. S., Borovskich and Yu. V.} (1994)
 \textit{Theory of U-statistics.} Mathematics and its Applications, 273. Kluwer Academic Publishers Group, Dordrecht.

\bibitem[KT12]{kusuotud}
\textsc{Kusuoka, S.,  Tudor, C. A.} (2012)
\newblock Stein's method for
invariant measures of diffusions via Malliavin calculus.
\newblock \emph{Stoch. Proc. Appl.} $\mathbf{122}$, 1627--1651.

\bibitem[Lam01]{Lam-Book}
\textsc{Lam, T. Y.} (2001)
\newblock  \emph{A First Course in Noncommutative Rings}. 
\newblock Second edition. Graduate Texts in Mathematics, 131. Springer-Verlag, New York.

\bibitem[Lee90]{lee-Ustatistics}
\textsc{Lee, A. J.} (1990)
 \textit{U-statistics. Theory and practice.}  Statistics: Textbooks and Monographs, 110. Marcel Dekker, Inc., New York.
	
\bibitem[LRS17]{ley} 
\textsc{Ley, C., Reinert, G., Swan, Y.} (2017)
\newblock  Stein's method for comparison of univariate distributions.  \newblock \emph{Probab. Surv.} $\mathbf{14}$,  1--52.		

\bibitem[Luk70]{lukas}
\textsc{Lukacs, E.} (1970).
\newblock {\em Characteristic functions}.
\newblock Hafner Publishing Co., New York.
\newblock Second edition, revised and enlarged.

\bibitem[MR87]{mr-Book87}
\textsc{McConell, J. C.,  Robson, J. C.} (1987)
\newblock \emph{Noncommutative Noetherian Rings}.  
\newblock Wiley-Interscience, Chichester.

\bibitem[MRRS23]{mrs21}
\textsc{Mijoule, G., Rai\v{c}, M., Reinert, G., Swan, Y.} (2023)
\newblock Stein's density method for multivariate continuous distributions.
\newblock \emph{Electron. J. Probab.} $\mathbf{23}$, No.\ 59, 1--40.

\bibitem[NN18]{n-n-book}
\textsc{Nualart, D., Nualart, E.} (2018)
\textit{Introduction to Malliavin Calculus.}  Institute of Mathematical Statistics Textbooks. Cambridge University Press.

\bibitem[NP09]{np09} 
\textsc{Nourdin, I., Peccati, G.} (2009)
\newblock Stein's method on Wiener chaos. 
\newblock \emph{Probab. Th. Rel. Fields} $\mathbf{145}$, 75--118.

\bibitem[NP12]{n-p-book}
\textsc{Nourdin, I., Peccati, G.} (2012)
\newblock \emph{Normal Approximations Using Malliavin Calculus: from Stein's Method to Universality}.
\newblock  Cambridge Tracts in Mathematics. Cambridge University Press.


\bibitem[Ore33]{Ore}
\textsc{Ore, O.} (1933) 
\newblock Theory of non-commutative polynomials.
\newblock \emph{Ann. Math.} $\mathbf{34}$, 480--508.

\bibitem[Sta78]{Stafford-2generators}
\textsc{Stafford, J. T.} (1978)
\newblock Module structure of Weyl algebras.
\newblock \emph{J. London Math. Soc.} (2) 18, 429--442.
	
\bibitem[Sta87]{stafford-87}
\textsc{Stafford,  J. T.} (1987)
\newblock Endomorphisms of Right Ideals of the Weyl Algebra.
\newblock \emph{T. Amer. Math. Soc.} $\mathbf{299}$, 623--639.

\bibitem[Sta98]{Stanley-Book}
\textsc{Stanley, R. P.} (1998)
\newblock \emph{Enumerative Combinatorics.} 
\newblock Vol. II, Cambridge University Press.

\bibitem[Ste72]{stein} 
\textsc{Stein, C.} (1972)
\newblock A bound for the error in the normal approximation to the the distribution of a sum of dependent random variables.
\newblock  In \emph{Proc. Sixth Berkeley Symp. Math. Statist. Probab.}  vol.\ 2, University of California Press, Berkeley,  583--602.


\bibitem[Was65]{Wasow-Book}
\textsc{Wasow, W.} (1987). 
\newblock \emph{Asymptotic Expansions for Ordinary Differential Equations}. \newblock Dover, 1987, A reprint of the John Wiley edition, 1965.

\bibitem[Zei90]{Zeilberger90}
\textsc{Zeilberger, D.} (1990)
\newblock A holonomic approach to special functions identities. 
\newblock \emph{J. Comput. Appl. Math.} $\mathbf{32}$, 321--368.
\end{thebibliography}
\end{document}